\documentclass[a4paper,oneside,11pt]{article}

\usepackage[T1]{fontenc}             
\usepackage[french,english]{babel}           %
\usepackage[utf8]{inputenc}        
\usepackage{amsmath,amssymb,dsfont}  
\usepackage{textcomp}                
\usepackage{listings,multicol}       
\usepackage{graphicx,color}          
\usepackage{wrapfig}                 
\usepackage[rightcaption]{sidecap}   
\usepackage{fullpage}                
\usepackage{a4wide}                  
\usepackage{amsthm}
\usepackage{ulem}                    
\usepackage{fancyhdr}                
\usepackage{subeqnarray}
\usepackage{graphicx}
\usepackage{caption}
\usepackage{subcaption}
\usepackage{rotating}
\usepackage{pdflscape}
\usepackage{enumerate}

\newtheorem{thm}{Theorem}[section]
\newtheorem{prop}{Proposition}[section]
\newtheorem{lm}{Lemma}[section]
\newtheorem{coro}{Corollary}[section]

\theoremstyle{definition}

\theoremstyle{remark}

\newcommand{\pa}[1]{\ensuremath{\left( #1 \right)}}
\newcommand{\cro}[1]{\ensuremath{\left[ #1 \right]}}
\newcommand{\ac}[1]{\ensuremath{\left\{ #1 \right\}}}
\newcommand{\abs}[1]{\ensuremath{\left| #1 \right|}}

\newcommand{\argmin}{\operatornamewithlimits{argmin}}

\newcommand{\R}{\ensuremath{\mathds{R}}}

\newcommand{\X}{\ensuremath{\mathds{X}}}

\renewcommand{\L}{\ensuremath{\mathds{L}}}
\newcommand{\calX}{\ensuremath{\mathcal{X}}}

\newcommand{\hyp}{\ensuremath{\mathcal{H}}}

\newcommand{\ds}[1]{\ensuremath{\mathds{#1}}} 			
\newcommand{\mc}[1]{\ensuremath{\mathcal{#1}}}          
\newcommand{\mfk}[1]{\ensuremath{\mathfrak{#1}}}		

\newcommand{\proba}[1]{\ensuremath{\mathds{P}\!\left(#1\right)}}			
\newcommand{\esp}[1]{\ensuremath{\mathds{E}\left[  #1 \right]}} 		
\newcommand{\Estar}[1]{\ensuremath{\mathds{E}^*\!\!\left[  #1 \right]}} 
\newcommand{\esps}[2]{\ensuremath{\mathds{E}_{#1}\!\left[ #2 \right]}}	

\newcommand{\var}[1]{\ensuremath{\operatorname{Var}\!\left(  #1 \right)}} 		
\newcommand{\Vstar}[1]{\ensuremath{\operatorname{Var}^*\!\!\left(  #1 \right)}}
\newcommand{\vars}[2]{\ensuremath{\operatorname{Var}_{#1}\!\left(  #2 \right)}} 		
\newcommand{\1}[1]{\ensuremath{\mathds{1}_{ #1}}}		

\newcommand{\PP}{\ensuremath{\mathds{P}}}

\newcommand{\norm}[1]{\| #1 \|}
\newcommand{\perm}{\ensuremath{{\pi_n}}}
\newcommand{\rperm}{\ensuremath{{\Pi_n}}}

\newcommand{\cv}[1]{\underset{#1}{\longrightarrow}}

\newcommand{\indep}{\ensuremath{\perp\!\!\!\perp}}
\newcommand{\Sn}[1]{\ensuremath{\mathfrak{S}_{#1}}}

\definecolor{orange}{cmyk}{0,0.5,1,0.3}

\newcommand{\ca}{\ensuremath{\mfk{a}}}
\newcommand{\cc}{\ensuremath{\mfk{c}}}

\newcommand{\med}[1]{\operatorname{med}\pa{#1}}

\begin{document}

\title{Concentration inequalities for randomly permuted sums}
\author{M\'elisande Albert\thanks{Institut de Mathématiques de Toulouse ; UMR5219. 
Université de Toulouse ; CNRS.
INSA IMT, F-31077 Toulouse, France.}}
\date{}
\maketitle

\renewcommand{\abstractname}{\vspace{-1cm}}
\begin{abstract} 
Initially motivated by the study of the non-asymptotic properties of non-parametric tests based on permutation methods, concentration inequalities for uniformly permuted sums have been largely studied in the literature. Recently, Delyon et al. proved a new Bernstein-type concentration inequality based on martingale theory. 
This work presents a new proof of this inequality based on the fundamental inequalities for random permutations of Talagrand. The idea is to first obtain a rough inequality for the square root of the permuted sum, and then, iterate the previous analysis and plug this first inequality to obtain a general concentration of permuted sums around their median. Then, concentration inequalities around the mean are deduced. This method allows us to obtain the Bernstein-type inequality up to constants, and, in particular, to recovers the Gaussian behavior of such permuted sums under classical conditions encountered in the literature. 
Then, an application to the study of the second kind error rate of permutation tests of independence is presented. 
\end{abstract}

\bigskip


\medskip


\section{Introduction and motivation}
\label{sct:intro}

This article presents concentration inequalities for randomly permuted sums defined by 
$Z_n=\sum_{i=1}^{n} a_{i,\rperm(i)},$ where $\ac{a_{i,j}}_{1\leq i,j\leq n}$ are real numbers, and $\rperm$ is a uniformly distributed random permutation of the set $\ac{1,\dots,n}$. Initially motivated by hypothesis testing in the non-parametric framework (see \cite{WaldWolfowitz1944} for instance), such sums have been largely studied from an asymptotic point of view in the literature. 
A first combinatorial central limit theorem is proved by Wald and Wolfowitz in \cite{WaldWolfowitz1944}, in the particular case when the real numbers $a_{i,j}$ are of a product form $b_i \times c_j$, under strong assumptions that have been released for instance by Noether \cite{Noether1949}. 
Then, Hoeffding obtains stronger results in such product case, and generalizes those results to not necessarily product type real terms $a_{i,j}$ in \cite{Hoeffding1951}. More precisely, he considers 
\begin{equation}
\label{eq:def_dij}
d_{i,j} = a_{i,j} - \frac{1}{n} \sum_{k=1}^n a_{k,j} -\frac{1}{n} \sum_{l=1}^n a_{i,l} + \frac{1}{n^2} \sum_{k,l=1}^n a_{k,l}. 
\end{equation}
In particular, $\var{Z_n} = \frac{1}{n-1}\sum_{i=1}^n d_{i,j}^2$. 
Then he proves (see \cite[Theorem 3]{Hoeffding1951}) that, if 
\begin{equation}
\label{eq:cond_Hoeffding-mieux}
\lim_{n\to+\infty} \frac{\frac{1}{n}\sum_{1\leq i,j \leq n} d_{i,j}^r}{\pa{\frac{1}{n}\sum_{i,j=1}^n d_{i,j}^2}^{r/2}} = 0, \quad \mbox{for some }r>2, 
\end{equation}
then the distribution of $Z_n = \sum_{i=1}^n a_{i,\rperm(i)}$ is asymptotically normal, that is, for all $x$ in $\R$, 
$$ \lim_{n\to +\infty} \proba{Z_n - \esp{Z_n} \leq x \sqrt{\var{Z_n}}} = \frac{1}{\sqrt{2\pi}} \int_{-\infty}^x e^{-\frac{y^2}{2}} dy.$$
He also considers a stronger (in the sense that it implies \eqref{eq:cond_Hoeffding-mieux}), but simpler condition in \cite[Theorem 3]{Hoeffding1951}, precisely 
\begin{equation}
\label{eq:cond_Hoeffding-pareil}
\frac{\max_{1\leq i,j\leq n} \ac{\abs{d_{i,j}}}}{\sqrt{\frac{1}{n} \sum_{i,j=1}^n d_{i,j}^2}}\cv{n\to+\infty} 0, 
\end{equation}
under which such an asymptotic Gaussian limit holds. 
Similar results have been obtained later, for instance by Motoo \cite{Motoo1956}, under the following Lindeberg-type condition that is for all $\varepsilon >0$, 
\begin{equation}
\label{eq:cond_Lindeberg}
\lim_{n\to+\infty} \sum_{1\leq i,j\leq n} \pa{\frac{d_{i,j}}{d}}^2 \1{\abs{\frac{d_{i,j}}{d}}>\varepsilon} = 0,
\end{equation} 
where $d^2 = n^{-1} \sum_{1\leq i,j\leq n} d_{i,j}^2$. 
In particular, he proves in \cite{Motoo1956} that such Lindeberg-type condition is weaker than Hoeffding's ones in the sense that \eqref{eq:cond_Lindeberg} is implied by \eqref{eq:cond_Hoeffding-mieux} (and thus by \eqref{eq:cond_Hoeffding-pareil}). A few years later, Hájek \cite{Hajek1961} proves in the product case, that the condition \eqref{eq:cond_Lindeberg} is in fact necessary. 
A simpler proof of the sufficiency of the Lindeberg-type condition is given by Schneller \cite{Schneller1988} based on Stein's method. 

Afterwards, the next step was to study the convergence of the conditional distribution when the terms $a_{i,j}$ in the general case, or $b_i\times c_j$ in the product case, are random. Notably, Dwass studies in \cite{Dwass1955} the limit of the randomly permuted sum in the product case, where only the $c_j$'s are random, and proves that the conditional distribution given the $c_j$'s converges almost surely (a.s.) to a Gaussian distribution. Then, Shapiro and Hubert \cite{ShapiroHubert1979} generalized this study to weighted $U$-statistics of the form $\sum_{i\neq j}b_{i,j} h(X_i,X_j)$ where the $X_i$'s are independent and identically distributed (i.i.d.) random variables. In a first time, they show some a.s. asymptotic normality of this statistic. In a second time, they complete Jogdeo's \cite{Jogdeo1968} work in the deterministic case, proving asymptotic normality of permuted statistics based on the previous weighted $U$-statistic. More precisely, they consider the rank statistic $\sum_{i\neq j}b_{i,j} h(X_{R_i},X_{R_j})$, where $R_i$ is the rank of $V_i$ in a sample $V_1,\dots,V_n$ of i.i.d. random variables with a continuous distribution function. In particular, notice that considering such rank statistics is equivalent to considering uniformly permuted statistics.  
In \cite{AlbertBouretFromontReynaud2015}, the previous combinatorial central limit theorems is generalized to permuted sums of non-i.i.d. random variables $\sum_{i=1}^n Y_{i,\rperm(i)}$, for particular forms of random variables $Y_{i,j}$.
The main difference with the previous results comes from the fact that the random variables $Y_{i,j}$ are not necessarily exchangeable. 

Hence, the asymptotic behavior of permuted sums have been vastly investigated in the literature, allowing to deduce good properties for permutation tests based on such statistics, like the asymptotic size, or the power (see for instance \cite{Romano1989} or \cite{AlbertBouretFromontReynaud2015}). Yet, such results are purely asymptotic, while, in many application fields, such as neurosciences for instance as described in \cite{AlbertBouretFromontReynaud2015}, few exploitable data are available. Hence, such asymptotic results may not be sufficient. This is why a non-asymptotic approach is preferred here, leading to concentration inequalities.

\paragraph{}
Concentration inequalities have been vastly investigated in the literature, and the interested reader can refer to the books of Ledoux \cite{Ledoux2005}, Massart \cite{Massart2007}, or the more recent one of Boucheron, Lugosi, and Massart \cite{BoucheronLugosiMassart2013} for some overall reviews. Yet in many cases, they provide precise tail bounds for well-behaved functions or sums of independent random variables. For instance, let us recall the classical Bernstein inequality stated for instance in \cite[Proposition 2.9 and Corollary 2.10]{Massart2007}.
\begin{thm}[Bernstein's inequality, Massart 2007] 
\label{thm:Bernstein_ineq_Massart07}
Let $X_1,\dots,X_n$ be independent real valued random variables. Assume that there exists some positive numbers $v$ and $c$ such that 
$$
\sum_{i=1}^n \esp{X_i^2}\leq v,
$$ 
and for all integers $k\geq 3$, 
$$
\sum_{i=1}^n \esp{{(X_i)}_+^k} \leq \frac{k!}{2} vc^{k-2},
$$
where $(\cdot)_+=\max\{\cdot,0\}$ denotes the positive part.  

Let $S =\sum_{i=1}^n (X_i-\esp{X_i})$, then for every positive $x$, 
\begin{equation}
\label{eq:ineq_Bernstein_stat}
\proba{S\geq \sqrt{2vx} + cx} \leq e^{-x}.
\end{equation}
Moreover, for any positive $x$, 
\begin{equation}
\label{eq:ineq_Bernstein_proba}
\proba{S\geq x} \leq \exp\pa{-\frac{x^2}{2(v+cx)}}.
\end{equation}
\end{thm}
Notice that both forms of Bernstein's inequality appear in the literature. Yet, due to its form, \eqref{eq:ineq_Bernstein_stat} is rather preferred in statistics, even though \eqref{eq:ineq_Bernstein_proba} is more classical.

The work in this article is based on the pioneering work of Talagrand (see \cite{Talagrand1995} for a review) who investigates the concentration of measure phenomenon for product measures. Of main interest here, he proved the following inequality for random permutations in \cite[Theorem~5.1]{Talagrand1995}.
\begin{thm}[Talagrand, 1995]
\label{thm:Talagrand}
Denote by $\Sn{n}$ the set of all permutations of $\ac{1,\dots,n}$. Define for any subset $A\subset \Sn{n}$, and permutation $\perm\in\Sn{n}$, 
$$U_A(\perm) = \ac{s \in \{0,1\}^n \ ;\ \exists \tau \in A \mbox{ such that } \forall 1\leq i\leq n,\ s_i=0 \implies \tau(i)=\perm(i)}.$$
Then, consider $V_A(\perm)=\operatorname{ConvexHull}\pa{U_A(\perm)}$, and 
$$f(A,\perm) = \min\ac{\sum_{i=1}^n v_i^2\ ;\ v=\pa{v_i}_{1\leq i\leq n}\in V_A(\perm)}.$$

Then, if $P_n$ denotes the uniform distribution on $\Sn{n}$, 
$$\int_{\Sn{n}} e^{\frac{1}{16} f(A,\perm)}dP_n(\perm) \leq \frac{1}{P_n(A)}.$$

Therefore, by Markov's inequality, for all $t>0$, 
\begin{equation}
\label{eq:Talagrand-Markov}
P_n\pa{\perm\ ;\ f(A,\perm) \geq t^2} \leq \frac{e^{-t^2/16}}{P_n(A)}.
\end{equation}
\end{thm}

This result on random permutations is fundamental, and is a key point to many other non-asymptotic works on random permutations. 
Among them emerges McDiarmid's article \cite{McDiarmid2002} in which he derives from Talagrand's inequality, exponential concentration inequalities around the median for randomly permuted functions of the observation under Lipschitz-type conditions and applied to randomized methods for graph coloring. 
More recently, Adamczak et al. obtained in \cite{AdamczakChafaiWolff2014} some concentration inequality under convex-Lipschitz conditions when studying the empirical spectral distribution of random matrices. 
In particular, they prove the following Theorem (precisely \cite[Theorem 3.1]{AdamczakChafaiWolff2014}). 
\begin{thm}[Adamczak, Chafai and Wolff, 2014] 
\label{thm:AdamczakChafaiWolff}
Consider $x_1,\ldots,x_n$ in $[0,1]$ and let $\varphi:[0,1]^n \to \R$ be an $L$-Lipschitz convex function. Let $\rperm$ be a random uniform permutation of the set $\ac{1,\dots,n}$ and denote $Y=\varphi\!\pa{x_{\rperm(1)},\dots, x_{\rperm(n)}}$. 
Then, there exists some positive absolute constant $c$ such that, for all $t>0$, 
$$\proba{Y-\esp{Y}\geq t} \leq 2\exp\pa{-\frac{ct^2}{L^2}}.$$
\end{thm}

Yet, the Lispchitz assumptions may be very restrictive and may not be satisfied by the functions considered in the application fields (see Section \ref{sct:motivstat} for instance).
Hence, the idea is to exploit the attractive form of a sum. 
Based on Stein's method, initially introduced to study the Gaussian behavior of sums of dependent random variables, Chatterjee studies permuted sums of non-negative numbers in \cite{Chatterjee2007}. He obtains in \cite[Proposition 1.1]{Chatterjee2007} the following first Bernstein-type concentration inequality for non-negative terms around the mean. 
\begin{thm}[Chatterjee, 2007]
\label{thm:Chatterjee2007}
Let $\ac{a_{i,j}}_{1\leq i,j\leq n}$ be a collection of numbers from $[0,1]$. 
Let $Z_n = \sum_{i=1}^n a_{i,\rperm(i)}$, where $\rperm$ is drawn from the uniform distribution over the set of all permutations of $\ac{1,\dots,n}$. Then, for any $t\geq 0$, 
\begin{equation}
\label{eq:Chatterjee2007}
\proba{\abs{Z_n-\esp{Z_n}}\geq t} \leq 2 \exp\pa{-\frac{t^2}{4\esp{Z_n} + 2t}}.
\end{equation}
\end{thm}
Notice that because of the expectation term in the right-hand side of \eqref{eq:Chatterjee2007}, the link with Hoeffding's combinatorial central limit theorem (for instance) is not so clear. \\

In \cite[Theorem 4.3]{BercuDelyonRio2015}, this result is sharpened in the sense that this expectation term is replaced by a variance term, allowing us to provide a non-asymptotic version of such combinatorial central limit theorem. This result is moreover generalized to any real numbers (not necessarily non-negative). 
More precisely, based on martingale theory, they prove the following result. 

\begin{thm}[Bercu, Delyon and Rio, 2015]
\label{thm:BercuDelyonRio2015}
Let $\ac{a_{i,j}}_{1\leq i,j\leq n}$ be an array of real numbers from $[-m_a,m_a]$. 
Let $Z_n = \sum_{i=1}^n a_{i,\rperm(i)}$, where $\rperm$ is drawn from the uniform distribution over the set of all permutations of $\ac{1,\dots,n}$. Then, for any $t> 0$, 
\begin{equation}
\label{eq:BercuDelyonRio2015}
\proba{\abs{Z_n-\esp{Z_n}}\geq t} \leq 4 \exp\pa{-\frac{t^2}{16(\theta \frac{1}{n}\sum_{i,j=1}^n a_{i,j}^2 + m_a t/3)}},
\end{equation}
where $\displaystyle\theta=\frac{5}{2} \ln(3) - \frac{2}{3}.$
\end{thm}
In this work, we obtain a similar result (up to constants) but based on a completely different approach. Moreover, this approach provides a direct proof for a concentration inequality of a permuted sum around its median. 

\paragraph{}
The present work is organized as follows. In Section \ref{sct:conc_ineq} are formulated the main results. 
Section \ref{sct:sum_pos_case} is devoted to the permuted sums of non-negative numbers. Based on Talagrand's result, a first rough concentration inequality for the square root of permuted sum is obtained in Lemma \ref{lm:concRac}. Then by iterating the previous analysis and plugging this first inequality, a general concentration of permuted sums around their median is obtained in Proposition \ref{prop:concPosMed}. Finally, the concentration inequality of Proposition \ref{prop:concPosMoy} around the mean is deduced. In Section \ref{sct:sum_general_case}, the previous inequalities are generalized to general permuted sums of not necessarily non-negative terms in Theorem \ref{thm:concQcqMoy}. 
Section \ref{sct:applitestindep} presents an application to the study of non-asymptotic properties of a permutation independence test in Statistics. In particular, a sharp control of the critical value of the test is deduced from the main result. 
The proofs are detailed in Section \ref{sct:proofs}. 
Finally, Appendix~\ref{sct:cond_Chebychev_Hoeffding_gen} contains technical results for the non-asymptotic control of the second kind error rate of the permutation test introduced in Section \ref{sct:applitestindep}.

\section{Bernstein-type concentration inequalities for permuted sums}
\label{sct:conc_ineq}

Let us first introduce some general notation. 
In the sequel, denote by $\Sn{n}$ the set of permutations of $\ac{1,2,\ldots,n}$. 
For all collection of real numbers $\ac{a_{i,j}}_{1\leq i,j\leq n}$, and for each $\perm$ in $\Sn{n}$, consider the permuted sum 
$$Z_n(\perm)=\sum_{i=1}^n a_{i,\perm(i)}.$$

Let $\rperm$ be a random uniform permutation in $\Sn{n}$, and $Z_n:=Z_n(\rperm)$. 
Denote $\med{Z_n}$ its median, that is which satisfies 
$$\proba{Z_n\geq \med{Z_n}} \geq 1/2\quad \mbox{and}\quad \proba{Z_n\leq \med{Z_n}}\geq 1/2. $$

This study is divided in two steps. The first one is restrained to non-negative terms. The second one extends the previous results to general terms, based on a trick involving both non-negative and negative parts.  

\subsection{Concentration of permuted sums of non-negative numbers}
\label{sct:sum_pos_case}

In the present section, the collection of numbers $\ac{a_{i,j}}_{1\leq i,j\leq n}$ is assumed to be non-negative. 
The proof of the concentration inequality around the median in Proposition \ref{prop:concPosMed} needs a preliminary step which is presented in Lemma \ref{lm:concRac}. 
It provides concentration inequality for the square root of the sum. 
It allows us then by iterating the same argument, and plugging the obtained inequality to the square root of the sum of the squares, namely $\sqrt{\sum_{i=1}^n a_{i,\rperm(i)}^2}$, to be able to sharpen Chatterjee's concentration inequality \eqref{eq:Chatterjee2007}. 

\begin{lm}
\label{lm:concRac}
Let $\ac{a_{i,j}}_{1\leq i,j\leq n}$ be a collection of non-negative numbers, and $\rperm$ be a random uniform permutation in $\Sn{n}$. Consider $Z_n=\sum_{i=1}^n a_{i,\rperm(i)}$. 
Then, for all $t>0$, 
\begin{equation}
\label{eq:concRac+}
\proba{\sqrt{Z_n} \geq \sqrt{\med{Z_n}} + t\sqrt{\max_{1\leq i,j\leq n}\ac{a_{i,j}}}}\leq 2 e^{-t^2/16}, 
\end{equation}
and 
\begin{equation}
\label{eq:concRac-}
\proba{\sqrt{Z_n} \leq \sqrt{\med{Z_n}} - t\sqrt{\max_{1\leq i,j\leq n}\ac{a_{i,j}}}}\leq 2 e^{-t^2/16}.
\end{equation}
In particular, one obtains the following two-sided concentration for the square root of a randomly permuted sum of non-negative numbers, 
$$\proba{\abs{\sqrt{Z_n} - \sqrt{\med{Z_n}}} > t\sqrt{\max_{1\leq i,j\leq n}\ac{a_{i,j}}}}\leq 4 e^{-t^2/16}.$$
\end{lm}

The idea of the proof is the same that the one of Adamczak et al. in \cite[Theorem 3.1]{AdamczakChafaiWolff2014}, but with a sum instead of a convex Lipschitz function. 
In a similar way, it is based on Talagrand's inequality for random permutations recalled in Theorem \ref{thm:Talagrand}.

\paragraph{}
In the following are presented two concentration inequalities in the non-negative case; the first one around the median, and the second one around the mean. It is well known that both are equivalent up to constants, but here, both are detailed in order to give the order of magnitude of the constants. The transition from the median to the mean can be obtained thanks to Ledoux' trick in the proof of \cite[Proposition 1.8]{Ledoux2005} allowing to reduce exponential concentration inequalities around any constant $m$ (corresponding in our case to $\med{Z_n}$) to similar inequalities around the mean. This trick consists in using the exponentially fast decrease around $m$ to upper bound the difference between $m$ and the mean. Yet, this approach leads to drastic multiplicative constants (of the order $8e^{16\pi}$ as shown in \cite{Albert2015PhD}). Better constants can be deduced from the following lemma. 
\begin{lm}
\label{lm:ineqEspMedVar}
For any real valued random variable $X$,
$$\abs{\esp{X}-\med{X}} \leq \sqrt{\var{X}}.$$
\end{lm}

In particular, we obtain the following result. 
\begin{prop}
\label{prop:concPosMed}
Let $\ac{a_{i,j}}_{1\leq i,j\leq n}$ be a collection of non-negative numbers and $\rperm$ be a random uniform permutation in $\Sn{n}$. Consider $Z_n=\sum_{i=1}^n a_{i,\rperm(i)}$. 
Then, for all $x>0$, 
\begin{equation}
\label{eq:concPosMed}
\proba{\abs{Z_n - \med{Z_n}} > \sqrt{\med{\sum_{i=1}^n a_{i,\rperm(i)}^2}x} + x\max_{1\leq i,j\leq n}\ac{a_{i,j}}}\leq 8 \exp\pa{\frac{-x}{16}}.
\end{equation}
\end{prop}

Since in many applications, the concentration around the mean is more adapted, the following proposition shows that one may obtain a similar behavior around the mean, at the cost of higher constants.

\begin{prop}
\label{prop:concPosMoy}
Let $\ac{a_{i,j}}_{1\leq i,j\leq n}$ be a collection of non-negative numbers, and $\rperm$ be a random uniform permutation in $\Sn{n}$. Consider $Z_n=\sum_{i=1}^n a_{i,\rperm(i)}$. 

Then, for all $x>0$, 
\begin{equation}
\label{eq:concPosMoy}
\proba{\abs{Z_n-\esp{Z_n}}\geq 2\sqrt{\pa{\frac{1}{n}\sum_{i,j=1}^n a_{i,j}^2}x} + \max_{1\leq i,j\leq n} \ac{a_{i,j}} x}\leq 8e^{1/16} \exp\pa{-\frac{x}{16}}.
\end{equation}
\end{prop}

This concentration inequality is called a Bernstein-type inequality restricted to non-negative sums, due to its resemblance to the standard Bernstein inequality, as recalled in Theorem \ref{thm:Bernstein_ineq_Massart07}. The main difference here lies in the fact that the random variables in the sum are not independent. 
Moreover, this inequality implies a more popular form of Bernstein's inequality stated in Corollary \ref{coro:concPosMoy}. 
\begin{coro}
\label{coro:concPosMoy}
With the same notation and assumptions as in Proposition \ref{prop:concPosMoy}, for all $t>0$, 
\begin{equation}
\label{eq:concPosMoy2}
\proba{\abs{Z_n-\esp{Z_n}}\geq t}\leq 8e^{1/16} \exp\pa{\frac{-t^2}{16\pa{4\frac{1}{n}\sum_{i,j=1}^n a_{i,j}^2 + 2\max_{1\leq i,j\leq n} \ac{a_{i,j}} t}}}.
\end{equation}
\end{coro}

\emph{Comment:}
Recall Chatterjee's result in \cite[Proposition 2.1]{Chatterjee2007}, quoted in Theorem \ref{thm:Chatterjee2007}, which can easily be rewritten with our notation, and for a collection of non-negative numbers not necessarily in $[0,1]$, by 
$$\forall t>0, \quad \proba{\abs{Z_n-\esp{Z_n}}\geq t}\leq 2\exp\pa{\frac{-t^2}{4M_a\frac{1}{n}\sum_{i,j=1}^n a_{i,j} + 2M_a t}},$$
where $M_a$ denotes the maximum $\max_{1\leq i,j\leq n} \ac{a_{i,j}}$. 
As mentioned in \cite{BercuDelyonRio2015}, the inequality in \eqref{eq:concPosMoy2} is sharper up to constants, thanks to the quadratic term since $\sum_{i,j=1}^n a_{i,j}^2 \leq M_a \sum_{i,j=1}^n a_{i,j}$ always holds.

\subsection{Concentration of permuted sums in the general case}
\label{sct:sum_general_case}

In this section, the collection of numbers $\ac{a_{i,j}}_{1\leq i,j\leq n}$ is no longer assumed to be non-negative. The following general concentration inequality for randomly permuted sums directly derives from Proposition \ref{prop:concPosMoy}. 
\begin{thm}
\label{thm:concQcqMoy}
Let $\ac{a_{i,j}}_{1\leq i,j\leq n}$ be a collection of any real numbers, and $\rperm$ be a random uniform permutation in $\Sn{n}$. Consider $Z_n=\sum_{i=1}^n a_{i,\rperm(i)}$. 
Then, for all $x>0$, 
\begin{equation}
\proba{\abs{Z_n-\esp{Z_n}}\geq 2\sqrt{2\pa{\frac{1}{n}\sum_{i,j=1}^n a_{i,j}^2}x} + 2\max_{1\leq i,j\leq n} \ac{\abs{a_{i,j}}} x}\leq 16e^{1/16} \exp\pa{-\frac{x}{16}}.
\end{equation}
\end{thm}

Once again, the obtained inequality is a Bernstein-type inequality. 
Moreover, it is also possible to obtain a more popular form of Bernstein-type inequalities applying the same trick based on the non-negative and the negative parts from Corollary \ref{coro:concPosMoy}. 
\begin{coro}
\label{coro:concQcqMoy}
With the same notation as in Theorem \ref{thm:concQcqMoy}, for all $t>0$, 
\begin{equation*}
\proba{\abs{Z_n-\esp{Z_n}}\geq t}\leq 16e^{1/16} \exp\pa{\frac{-t^2}{256\pa{\var{Z_n} + \max_{1\leq i,j\leq n} \ac{\abs{a_{i,j}}} t}}}.
\end{equation*}
\end{coro}

\emph{Comments:} One recovers a Gaussian behavior of the centered permuted sum obtained by Hoeffding in \cite[Theorem 3]{Hoeffding1951} under the same assumptions.
Indeed, in the proof of Corollary \ref{coro:concQcqMoy}, one obtains the following intermediate result (see \eqref{eq:concQcqMoyHoeffding}), that is 
$$\proba{\abs{Z_n-\esp{Z_n}}\geq t}\leq 16e^{1/16} \exp\pa{\frac{-t^2}{64\pa{4\frac{1}{n}\sum_{i,j=1}^n d_{i,j}^2 + \max_{1\leq i,j\leq n} \ac{\abs{d_{i,j}}} t}}},$$
where the $d_{i,j}$'s are defined in \eqref{eq:def_dij}. 
Yet, $\var{Z_n}=\frac{1}{n-1}\sum_{i,j=1}^n d_{i,j}^2$ (see \cite[Theorem 2]{Hoeffding1951}). 
Hence, applying this inequality to $t=x\sqrt{\var{Z_n}} \geq x\sqrt{n^{-1}\sum_{i,j=1}^n d_{i,j}^2}$ for $x>0$ leads to 
\begin{equation*}
\proba{\abs{Z_n-\esp{Z_n}}\geq x\sqrt{\var{Z_n}}}\leq 16e^{1/16} \exp\pa{\frac{-x^2}{256 \pa{1 + \frac{\max_{1\leq i,j\leq n} \ac{\abs{d_{i,j}}}}{\sqrt{\frac{1}{n} \sum_{i,j=1}^n d_{i,j}^2}} x}}}, 
\end{equation*}
Hence, under Hoeffding's simpler condition \eqref{eq:cond_Hoeffding-pareil}, namely
$$\lim_{n\to+\infty} \frac{\max_{1\leq i,j \leq n} d_{i,j}^2}{\frac{1}{n}\sum_{i,j=1}^n d_{i,j}^2} = 0, $$
one recovers, 
$$\lim_{n\to+\infty} \proba{\abs{Z_n-\esp{Z_n}}\geq x\sqrt{\var{Z_n}}} \leq 16e^{1/16} e^{-x^2/256},$$
which is a Gaussian tail that is, up to constants, close in spirit to the one obtained by Hoeffding in \cite[Theorem 3]{Hoeffding1951}.

\section{Application to independence testing}
\label{sct:applitestindep}
\subsection{Statistical motivation}
\label{sct:motivstat}

Let $\calX$ represent a separable set. 
Given an i.i.d. $n$-sample $\X_n = (X_1,\dots,X_n)$, where each $X_i$ is a couple $(X_i^1,X_i^2)$ in $\calX^2$ with distribution $P$ of marginals $P^1$ and $P^2$, we aim at testing 
\begin{center}
the null hypothesis $(\hyp_0)$ "$P=P^1\otimes P^2$" against the alternative $(\hyp_1)$ "$P\neq P^1\otimes P^2$".
\end{center} 
The considered test statistic is defined by 
\begin{equation}
\label{eq:def_stat_T}
T_\delta(\X_n) = \frac{1}{n-1} \pa{\sum_{i=1}^n \varphi_\delta (X_i^1,X_i^2) -\frac{1}{n} \sum_{i,j=1}^n \varphi_\delta (X_i^1,X_j^2)}, 
\end{equation}
where $\varphi_\delta$ is a measurable real-valued function on $\calX^2$ potentially depending on some unknown parameter $\delta$. 
Denoting for any real-valued measurable function $g$ on $\calX^2$, 
\begin{equation}
\label{eq:not_esps_general_indep}
\esps{P}{g} = \int_{\calX^2} g\!\pa{x^1,x^2}dP\!\pa{x^1,x^2}\quad \mbox{and}\quad \esps{\indep}{g} = \int_{\calX^2} g\!\pa{x^1,x^2}dP^1\!\pa{x^1} dP^2\!\pa{x^2},
\end{equation}
one may notice that, $T_\delta(\X_n)$ is an unbiased estimator of 
$$
\esp{T_\delta(\X_n)}=\esps{P}{\varphi_\delta} - \esps{\indep}{\varphi_\delta}, 
$$
which is equal to $0$ under $(\hyp_0)$. 
For more details on the choice of the test statistic, the interested reader can refer to \cite{AlbertBouretFromontReynaud2015} (motivated by synchrony detection in neuroscience for instance). The particular case where $\calX=[0,1]$ and the $\varphi_\delta$ are Haar wavelets is studied in \cite[Chaper 4]{Albert2015PhD}. Notice that in this case, the Lispchitz assumptions of Adamczak et al (see Theorem \ref{thm:AdamczakChafaiWolff}) are not satisfied, since the Haar wavelet functions are not even continuous. 

\paragraph{}
The critical value of the test is obtained from the permutation approach, inspired by Hoeffding \cite{Hoeffding1952}, and Romano \cite{Romano1989}. 
Let $\rperm$ be a uniformly distributed random permutation of $\ac{1,\dots,n}$ independent of $\X_n$ and consider the permuted sample 
$$\X_n^{\rperm} = (X_1^{\rperm}, \dots X_n^{\rperm}), \quad \mbox{where}\quad X_i^{\rperm}=(X_i^1,X_{\rperm(i)}^2) \ \forall 1\leq i\leq n,$$ 
obtained from permuting only the second coordinates.
Then, under $\hyp_0$, the original sample $\X_n$ and the permuted one $\X_n^{\rperm}$ have the same distribution. 
Hence, the critical value of the upper-tailed test, denoted by $q_{1-\alpha}(\X_n)$, is the $(1-\alpha)$-quantile of the conditional distribution of the permuted statistic $T_\delta(\X_n^{\rperm})$ given the sample $\X_n$, 
where the permuted test statistic is equal to
\begin{equation*}
T_\delta(\X_n^{\rperm}) = \frac{1}{n-1} \pa{\sum_{i=1}^n \varphi_\delta (X_i^1,X_{\rperm(i)}^2) -\frac{1}{n} \sum_{i,j=1}^n \varphi_\delta (X_i^1,X_j^2)}, 
\end{equation*}
More precisely, given $\X_n$, if 
$$T_\delta^{(1)}(\X_n)\leq T_\delta^{(2)}(\X_n)\leq \dots \leq T_\delta^{(n!)}(\X_n)$$
denote the ordered values of all the permuted test statistic $T_\delta(\X_n^{\perm})$, when $\perm$ describes the set of all permutations of $\ac{1,\dots,n}$, then the critical value is equal to 
\begin{equation}
\label{eq:def_crit_valq}
q_{1-\alpha}(\X_n) = T_\delta^{(n!-\lfloor n!\alpha \rfloor)}(\X_n). 
\end{equation}
The corresponding test rejects the null hypothesis when $T_\delta(\X_n) > q_{1-\alpha}(\X_n)$, here denoted by 
\begin{equation}
\label{eq:def_test_chap_conc}
\Delta_\alpha(\X_n) = \1{T_\delta(\X_n) > q_{1-\alpha}(\X_n)}.
\end{equation}

\paragraph{}
In \cite{AlbertBouretFromontReynaud2015}, the asymptotic properties of such test are studied. Based on a combinatorial central limit theorem in a non-i.i.d. case, the test is proved to be, under mild conditions, asymptotically of prescribed size, and power equal to one under any reasonable alternatives. 
Yet, as explained above, such purely asymptotic properties may be insufficient when applying these tests in neuroscience for instance. Moreover, the delicate choice of the parameter $\delta$ is a real question, especially, in neuroscience, where it has some biological meaning, as mentioned in \cite{AlbertBouretFromontReynaud2015} and \cite{AlbertBouretFromontReynaud2016}. 
A possible approach to overcome this issue is to aggregate several tests for different parameters $\delta$, and reject independence if at least one of them does. In particular, this approach should give us information on how to choose this parameter. Yet, to do so, non-asymptotic controls are necessary. 

\paragraph{}
From a non-asymptotic point of view, since the test is non-asymptotically of prescribed level by construction, remains the non-asymptotic control of the second kind error rate, that is the probability of wrongly accepting the null hypothesis. 
In the spirit of \cite{FromontLaurentReynaud2011,FromontLaurentReynaud2013,SansonnetTuleau2015}, the idea is to study the uniform separation rates, in order to study the optimality in the minimax sense (see \cite{Baraud2002}). 

From now on, consider an alternative $P$ satisfying $(\hyp_1)$, and an i.i.d. sample $\X_n$ from such distribution $P$. 
Assume moreover that the alternative satisfies $\esps{P}{\varphi_\delta}>\esps{\indep}{\varphi_\delta}$, that is $\esp{T_\delta(\X_n)} > 0$. 
The initial step is to find some condition on $P$ guaranteeing the control of the second kind error rate, namely $\proba{\Delta_\alpha(\X_n) = 0}$, by a prescribed value $\beta>0$. 
Intuitively, since the expectation of the test statistic $\esp{T_\delta(\X_n)}$ is equal to zero under the null hypothesis, the test should be more efficient in rejecting $(\hyp_0)$ for large values of this expectation. 
So, the aim is to find conditions of the form $\esp{T_\delta(\X_n)}\geq s$ for some threshold $s$ to be determined. 
Yet, one of the main difficulties here comes from the randomness of the critical value. 
The idea, as in \cite{FromontLaurentReynaud2011}, is thus to introduce $q^\alpha_{1-\beta/2}$ the $(1-\beta/2)$-quantile of the critical value $q_{1-\alpha}(\X_n)$ and deduce from Chebychev's inequality (see Appendix \ref{sct:cond_Chebychev}), that the second kind error rate is controlled by $\beta$ as soon as 
\begin{equation}
\label{eq:cond_simple_tchebychev}
\esp{T_\delta(\X_n)} \geq q^\alpha_{1-\beta/2} + \sqrt{\frac{2}{\beta} \var{T_\delta(\X_n)}}.
\end{equation}
Usually, the goal in general minimax approaches is to express, for well-chosen functions $\varphi_\delta$, some distance between the alternative $P$ and the null hypothesis $(\hyp_0)$ thanks to $\esp{T_\delta(\X_n)}$ for which minimax lower-bounds are known (see for instance \cite{FromontLaurentReynaud2011,FromontLaurentReynaud2013}). 
The objective is then to control, up to a constant, such distance (and in particular each term in the right-hand side of \eqref{eq:cond_simple_tchebychev}) by the minimax rate of independence testing with respect to such distance on well-chosen regularity subspaces of alternatives, in order to prove the optimality of the method from a theoretical point of view. 
The interested reader could refer to the thesis \cite[Chapter 4]{Albert2015PhD} for more details about this kind of development in the density case.  
It is not in the scope of the present article to develop such minimax theory in the general case, but to provide some general tools providing some sharp control of each term in the right-hand side of \eqref{eq:cond_simple_tchebychev} which consists in a very first step of this approach. 
Some technical computations imply that the variance term can be upper bounded, up to a multiplicative constant, by $n^{-1} (\esp{\varphi_\delta^2(X_1^1,X_1^2)} + \esp{\varphi_\delta^2(X_1^1,X_2^2)})$ (see Lemma~\ref{lm:calcul_variance}). 
Hence, the challenging part relies in the quantile term. At this point, several ideas have been explored. 

\subsection{Why concentration inequalities are necessary}
\label{sct:motivconc}

A first idea to control the conditional quantile of the permuted test statistic is based on the non-asymptotic control of the critical value obtained in Appendix \ref{sct:control_quantile_cond} (see equation \eqref{eq:maj_quantile_Hoeffding}), following Hoeffding's idea (see \cite[Theorem 2.1]{Hoeffding1952}), that leads to the condition
\begin{equation}
\label{eq:cond_Hoeffding}
\esp{T_\delta(\X_n)} \geq \frac{4}{\sqrt{\alpha}}\sqrt{\frac{2}{\beta}\frac{\esp{\varphi_\delta(X_1^1,X_1^2)^2} + \esp{\varphi_\delta(X_1^1,X_2^2)^2}}{n}}.
\end{equation}
The proof of this result is detailed in Appendix \ref{sct:cond_Hoeffding52}. 
Yet, this result may not be sharp enough, especially in $\alpha$. Indeed, as explained above, the next step consists in aggregating several tests for different values of the parameter $\delta$ in a purpose of adaptivity. 
Generally, when aggregating tests, as in multiple testing methods, the multiplicity of the tests has to be taken into account. In particular, the single prescribed level of each individual test should be corrected. 
Several corrections exist, such as the Bonferroni one, which consists in dividing the global desired level $\alpha$ by the number of tests $M$. Yet, for such correction, the lower-bound in \eqref{eq:cond_Hoeffding} comes with a cost in $\sqrt{M}$, which is too large to provide optimal rates. Even with more sophisticated corrections than the Bonferroni one (see, e.g., \cite{FromontLaurentReynaud2011,FromontLaurentReynaud2013,SansonnetTuleau2015}), 
the control by a term of order $\alpha^{-1/2}$ is too large, since classically in the literature, the dependence in $\alpha$ should be of the order of $\sqrt{\ln(1/\alpha)}$. 
Hence, the bound ensuing from this first track being not sharp enough, the next idea was to investigate other non-asymptotic approaches for permuted sums.

\paragraph{}
Such approaches have also been studied in the literature. 
For instance, Ho and Chen \cite{HoChen1978} obtain non-asymptotic Berry-Esseen type bounds in the $\L^p$-distance between the cumulative distribution function (c.d.f.) of the standardized permuted sum of i.i.d. random variables and the c.d.f. of the normal distribution, based on Stein's method. In particular, they obtain the rate of convergence to a normal distribution in $\L^p$-distance under Lindeberg-type conditions. Then, Bolthausen \cite{Bolthausen1984} considers a different approach, also based on Stein's method allowing to extend Ho and Chen's results in the non-identically distributed case. More precisely, he obtains bounds in the $\L^\infty$-distance in the non-random case. 
In particular, in the deterministic case (which can easily be generalized to random cases), considering the notation introduced above, he obtains the following non-asymptotic bound:
\begin{equation*}
\sup_{x\in\R} \abs{\proba{Z_n -\esp{Z_n} \leq x\sqrt{\var{Z_n}}} - \Phi_{0,1}(x)} \leq \frac{C}{n\sqrt{\var{Z_n}}^3}\sum_{i,j=1}^n\abs{d_{i,j}}^3,
\end{equation*}
where $C$ is an absolute constant, and $\Phi_{0,1}$ denotes the standard normal distribution function. 
In particular, when applying this result to answer our motivation by considering random variables $\varphi_\delta(X_i^1,X_j^2)$ instead of the deterministic terms $a_{i,j}$, and working conditionally on the sample $\X_n$, the permuted statistic $T_\delta(\X_n^{\rperm})$ corresponds to $(n-1)^{-1}(Z_n-\esp{Z_n})$. 
Therefore, the previous inequality implies that, for all $t$ in $\R$,  
\begin{align}
\label{eq:cond_appl_Bolthausen}
\proba{T_\delta\!\pa{\X_n^{\rperm}} > t \middle| \X_n} \leq &\cro{1 - \Phi_{0,1}\pa{\frac{t}{\sqrt{\var{T_\delta\!\pa{\X_n^{\rperm}} \middle| \X_n}}}}} \nonumber\\
& +\quad \frac{C}{n(n-1)^{2/3}\sqrt{\var{T_\delta\!\pa{\X_n^{\rperm}} \middle| \X_n}}^3}\sum_{i,j}\abs{D_{i,j}}^3,
\end{align}
where $D_{i,j}$ denotes 
$$\varphi_\delta(X_i^1,X_j^2) -\frac{1}{n}\sum_{l=1}^n \varphi_\delta(X_i^1,X_l^2) -\frac{1}{n}\sum_{k=1}^n \varphi_\delta(X_k^1,X_j^2) + \frac{1}{n^2}\sum_{k,l=1}^n \varphi_\delta(X_k^1,X_l^2).$$
Yet, by definition of conditional quantiles, the critical value $q_{1-\alpha}(\X_n)$ is the smallest value of $t$ such that $\proba{T_\delta(\X_n) > t \middle| \X_n} \leq \alpha$.
Hence, considering \eqref{eq:cond_appl_Bolthausen}, one can easily make the first term of the sum in the right-hand side of the inequality as small as one wants by choosing $t$ large enough. 
However, the second term being fixed, nothing guarantees that the upper-bound in \eqref{eq:cond_appl_Bolthausen} can be constrained to be smaller than $\alpha$. Thus, this result cannot be applied in order to control non-asymptotically the critical value. 
Concentration inequalities seem thus to be adequate here, as they provide sharp non-asymptotic results, with usually exponentially small controls which leads to the desired logarithmic dependency in $\alpha$, as mentioned above.

\subsection{A sharp control of the conditional quantile and a new condition guaranteeing a control of the second kind error rate}
\label{sct:quantcontr}

Sharp controls of the quantiles are provided in the following proposition. 
\begin{prop}
\label{prop:control_quantile}
Consider the same notation as in Section \ref{sct:motivstat} and let $q_{1-\beta/2}^{\alpha}$ be the $(1-\beta/2)$-quantile of the conditional quantile $q_{1-\alpha}(\X_n)$.
Then, there exists two universal positive constants $C'$ and $c_0$ such that
\begin{equation}
\label{eq:control_conditional_quantile}
q_{1-\alpha}(\X_n) \leq \frac{C'}{n-1} \ac{\sqrt{\frac{1}{n} \sum_{i,j=1}^n\varphi_\delta^2(X_i^1,X_j^2)} \sqrt{\ln\pa{\frac{c_0}{\alpha}}} + \norm{\varphi_\delta}_{\infty}\ln\pa{\frac{c_0}{\alpha}}}. 
\end{equation}

As a consequence, there exists a universal positive constants $C$ such that
\begin{equation}
\label{eq:control_quantile}
q_{1-\beta/2}^\alpha \leq C \ac{\sqrt{\frac{2}{\beta}\ln\pa{\frac{c_0}{\alpha}}}\pa{\frac{\sqrt{\esps{P}{\varphi_\delta^2}}}{n} + \frac{\sqrt{\esps{\indep}{\varphi_\delta^2}}}{\sqrt{n}}}  + \frac{\norm{\varphi_\delta}_\infty}{n}\ln\pa{\frac{c_0}{\alpha}}}.
\end{equation}
\end{prop}

Moreover, a control of the variance term is obtained in the following lemma based on the Cauchy-Schwartz inequality. 
\begin{lm}
\label{lm:calcul_variance}
Let $n\geq 4$ and $\X_n$ be a sample of $n$ i.i.d. random variables with distribution $P$ and marginals $P^1$ and $P^2$. Let $T_\delta$ be the test statistic defined in \eqref{eq:def_stat_T}, and $\esps{P}{\cdot}$ and $\esps{\indep}{\cdot}$ be notation introduced in \eqref{eq:not_esps_general_indep}. Then, if both $\esps{P}{\varphi_\delta^2}<+\infty$ and $\esps{\indep}{\varphi_\delta^2}<+\infty$, 
$$\var{T_\delta(\X_n)} \leq \frac{1}{n} \pa{\sqrt{\esps{P}{\varphi_\delta^2}} + 2 \sqrt{\esps{\indep}{\varphi_\delta^2}}}^2.$$
\end{lm}

Proposition~\ref{prop:control_quantile} and Lemma \ref{lm:calcul_variance} both imply that the right-hand side of \eqref{eq:cond_simple_tchebychev} is upper bounded by 
\begin{equation}
\label{eq:upper_bound}
C''\ac{\sqrt{\frac{2}{\beta}\cro{\ln\pa{\frac{c_0}{\alpha}}+1}\frac{\pa{\esps{P}{\varphi_\delta^2} + \esps{\indep}{\varphi_\delta^2}}}{n}} + \frac{\norm{\varphi_\delta}_\infty}{n}\ln\pa{\frac{c_0}{\alpha}}},
\end{equation}
where $C''$ is a universal constant. 

Indeed, the control of $q_{1-\beta/2}^\alpha$ is implied by \eqref{eq:control_quantile} combined with the concavity property of the square-root function. 
Lemma \ref{lm:calcul_variance} directly implies that the variance term satisfies 
$$
\var{T_\delta(\X_n)} \leq \frac{8}{n} \pa{\esps{P}{\varphi_\delta^2} + \esps{\indep}{\varphi_\delta^2}},
$$

Finally, if $\esp{T_\delta(\X_n)}$ is larger than the quantity in \eqref{eq:upper_bound}, then condition \eqref{eq:cond_simple_tchebychev} is satisfied which directly provides that $\proba{\Delta_{\delta,\alpha}(\X_n) = 0} \leq \beta$, that is the second kind error rate of the test $\Delta_{\delta,\alpha}$ is less than or equal to the prescribed value $\beta$. 
One may notice that this time, the dependence in $\alpha$ is, as expected, of the order of $\sqrt{\ln(1/\alpha)}$.

\section{Proofs}
\label{sct:proofs}

\subsection{Proof of Lemma \ref{lm:concRac}}
\label{sct:lmConcRac}

\paragraph{Sketch of proof.} From now on, fix $t>0$.
Recall the notation introduced by Talagrand in Theorem \ref{thm:Talagrand}. The main purpose of these notation is to introduce some notion of distance between a permutation $\perm$ in $\Sn{n}$ and a subset $A$ of $\Sn{n}$. To do so, the idea is to reduce the set of interest to a simpler one, that is $[0,1]^n$, by considering $$U_A(\perm) = \ac{s \in \{0,1\}^n \ ;\ \exists \tau \in A \mbox{ such that } \forall 1\leq i\leq n,\ s_i=0 \implies \tau(i)=\perm(i)}.$$
One may notice that the permutation $\perm$ belongs to $A$ if and only if $0$ belongs to the set $U_A(\perm)$. 
Hence, the corresponding distance between the permutation $\perm$ and the set $A$ is coded by the distance between $0$ and the set $U_A(\perm)$ and thus defined by
$$f(A,\perm) = \min\ac{\sum_{i=1}^n v_i^2\ ;\ v=\pa{v_i}_{1\leq i\leq n}\in V_A(\perm)},$$
where $V_A(\perm)=\operatorname{ConvexHull}\pa{U_A(\perm)}$. 
One may notice in particular that $A$ contains $\perm$ if and only if the distance $f(A,\perm)=0$. 

The global frame of the proof of Lemma \ref{lm:concRac} (and also Proposition \ref{prop:concPosMed}) relies on the following steps. 
The first step consists in proving that
\begin{equation}
\label{eq:1step}
\proba{\sqrt{Z}\geq\sqrt{C_A} + t \sqrt{\max_{1\leq i,j\leq n}\ac{a_{i,j}}}} \leq \frac{e^{-t^2/16}}{\proba{Z\in A}},
\end{equation}
for some subset $A$ of $\Sn{n}$ of the shape $A=\ac{\tau\in\Sn{n} ; Z(\tau)\leq C_A}$ for some constant $C_A$ to be chosen later. 
For this purpose, since Talagrand's inequality for random permutations (see Theorem \ref{thm:Talagrand}) provides that 
$$\proba{f(A,\rperm) \geq t^2} \leq \frac{e^{-t^2/16}}{\proba{\rperm \in A}},$$
it is sufficient to prove that 
$$\proba{f(A,\rperm) \geq t^2}\geq \proba{\sqrt{Z}\geq\sqrt{C_A} + t \sqrt{\max_{1\leq i,j\leq n}\ac{a_{i,j}}}},$$ to obtain \eqref{eq:1step}. 
To do so, the idea, as in \cite{AdamczakChafaiWolff2014}, is to show that the assertion $f(A,\rperm) < t^2$ implies that $\sqrt{Z}< \sqrt{C_A} + t \sqrt{\max_{1\leq i,j\leq n}\ac{a_{i,j}}}$, and to conclude by contraposition. 

Then, the two following steps consist in choosing appropriate constants $C_A$ in \eqref{eq:1step} depending on the median of $Z$, such that both $\proba{\sqrt{Z}\geq\sqrt{C_A} + t \sqrt{\max_{1\leq i,j\leq n}\ac{a_{i,j}}}}$ and $\proba{Z\in A}$ are greater than $1/2$, in order to control both probabilities 
$$\proba{\sqrt{Z} \geq \sqrt{\med{Z}} + t\sqrt{\max_{1\leq i,j\leq n}\ac{a_{i,j}}}}\mbox{ and }\proba{\sqrt{Z} \leq \sqrt{\med{Z}} - t\sqrt{\max_{1\leq i,j\leq n}\ac{a_{i,j}}}}$$ respectively in \eqref{eq:concRac+} and \eqref{eq:concRac-}.

\paragraph{First step: preliminary study.}
Assume $f(A,\rperm)<t^2$. Then, by definition of the distance $f$, there exists some $s^1,\dots,s^m$ in $U_A(\rperm)$, and some non-negative weights $p_1,\dots,p_m $ satisfying $\sum_{j=1}^m p_j=1$ such that 
$$\sum_{i=1}^n \cro{\pa{\sum_{j=1}^m p_j s_i^j}^2} < t^2.$$
For each $1\leq j\leq m$, since $s^j$ belongs to $U_A(\rperm)$, one may consider a permutation $\tau_j$ in $A$ associated to $s^j$ (that is satisfying $s_i^j=0 \implies \tau_j(i)=\rperm(i)$). 
Then, since the $a_{i,j}$ are non-negative, and from the Cauchy-Schwartz inequality,
\begin{eqnarray*}
Z-\sum_{j=1}^m p_j Z(\tau_j) &=& \sum_{i=1}^n \sum_{j=1}^m p_j \pa{a_{i,\rperm(i)}-a_{i,\tau_j(i)}}\\
&=& \sum_{i=1}^n \sum_{j=1}^m p_j \pa{a_{i,\rperm(i)}-a_{i,\tau_j(i)}}s_i^j\\ 
&\leq & \sum_{i=1}^n \cro{\pa{\sum_{j=1}^m p_j s_i^j} a_{i,\rperm(i)}} \\
&\leq & \sqrt{\sum_{i=1}^n \pa{\sum_{j=1}^m p_j s_i^j}^2} \sqrt{\sum_{i=1}^n a_{i,\rperm(i)}^2}\\
&< & t \sqrt{\max_{1\leq i,j\leq n}\ac{a_{i,j}}} \sqrt{Z}.
\end{eqnarray*}
Thus, as the $\tau_j$ are in $A=\ac{\tau ; Z(\tau)\leq C_A}$, 
$$ Z < C_A + t \sqrt{\max_{1\leq i,j\leq n}\ac{a_{i,j}}} \sqrt{Z}.$$

Therefore, by solving the second-order polynomial in $\sqrt{Z}$ above, one obtains
$$ \sqrt{Z} < \frac{t \sqrt{\max_{1\leq i,j\leq n}\ac{a_{i,j}}} + \sqrt{t^2\max_{1\leq i,j\leq n}\ac{a_{i,j}}+4C_A}}{2}\leq t \sqrt{\max_{1\leq i,j\leq n}\ac{a_{i,j}}} + \sqrt{C_A}. $$

Finally, by contraposition, 
$$\proba{\sqrt{Z}\geq\sqrt{C_A} + t \sqrt{\max_{1\leq i,j\leq n}\ac{a_{i,j}}}} \leq \proba{f(A,\rperm)\geq t^2},$$ 
which, combined with \eqref{eq:Talagrand-Markov} of Theorem \ref{thm:Talagrand} provides \eqref{eq:1step}.

\paragraph{Second step: proof of \eqref{eq:concRac+}.}
Taking $C_A=\med{Z}$ guarantees $\proba{Z\in A}\geq 1/2$ and thus, \eqref{eq:1step} provides \eqref{eq:concRac+}.

\paragraph{Third step: proof of \eqref{eq:concRac-}.}
Taking $C_A = \pa{\sqrt{\med{Z}} - t\sqrt{\max_{1\leq i,j\leq n}\ac{a_{i,j}}}}^2$ implies 
$$\proba{\sqrt{Z}\geq \sqrt{C_A} + t \sqrt{\max_{1\leq i,j\leq n}\ac{a_{i,j}}}} = \proba{\sqrt{Z}\geq\sqrt{\med{Z}}}=\proba{Z\geq\med{Z}} \geq \frac{1}{2}.$$

So finally, again by \eqref{eq:1step},
\begin{eqnarray*}
\proba{\sqrt{Z} \leq \sqrt{\med{Z}} - t\sqrt{\max_{1\leq i,j\leq n}\ac{a_{i,j}}}} 
& = & \proba{Z\in A} \\
&\leq & \frac{e^{-t^2/16}}{\proba{\sqrt{Z}\geq \sqrt{C_A} + t \sqrt{\max_{1\leq i,j\leq n}\ac{a_{i,j}}}}}\\
&\leq & 2 e^{-t^2/16},
\end{eqnarray*}
which ends the proof of the Lemma.

\subsection{Proof of Proposition \ref{prop:concPosMed}}
\label{sct:propConcPosMed}

From now on, fix $x>0$, and consider $t=x^2$.
This proof is again based on Talagrand's inequality for random permutations, combined with \eqref{eq:concRac+} in Lemma \ref{lm:concRac}. It follows exactly the same progression as in the proof of Lemma \ref{lm:concRac}; 
the preliminary step consists in working with subsets $A\subset \Sn{n}$ of the form $A=\ac{\tau\in\Sn{n}\ ;\ Z(\tau)\leq C_A}$ for some constant $C_A$, in order to obtain for all $v>0$,
\begin{equation}
\label{eq:1stepPos}
\proba{Z \geq C_A + t \pa{\sqrt{\med{\sum_{i=1}^n a_{i,\rperm(i)}^2}} + v\max_{1\leq i,j \leq n}\ac{a_{i,j}}}}\leq \frac{e^{-t^2/16}}{\proba{Z\in A}} + 2 e^{-v^2/16}.
\end{equation}
The second and third step consist in picking up a well-chosen constant $C_A$ and a well-chosen $v>0$ in order to obtain respectively
\begin{equation}
\label{eq:concPosMed+}
\proba{Z \geq \med{Z} + t\pa{\sqrt{\med{\sum_{i=1}^n a_{i,\rperm(i)}^2}} + (t\vee C_0) \max_{1\leq i,j\leq n}\ac{a_{i,j}}}}\leq 4 e^{-t^2/16}, 
\end{equation}
and 
\begin{equation}
\label{eq:concPosMed-}
\proba{Z \leq \med{Z} - t\pa{\sqrt{\med{\sum_{i=1}^n a_{i,\rperm(i)}^2}} + (t\vee C_0) \max_{1\leq i,j\leq n}\ac{a_{i,j}}}}\leq 4 e^{-t^2/16}, 
\end{equation}
where $C_0 = 4\sqrt{\ln(8)}$.
The final step combines \eqref{eq:concPosMed+} and \eqref{eq:concPosMed-} in order to prove \eqref{eq:concPosMed}.

\paragraph{First step: preliminary study.}
Let $A=\ac{\tau\in\Sn{n}\ ;\ Z(\tau)\leq C_A}$ with $C_A$ a general constant, and fix $v>0$. 
Assume, this time, that both
\begin{equation}
\label{eq:hypPos}
f(A,\rperm)<t^2 \quad\mbox{ and }\quad
\sqrt{\sum_{i=1}^n a_{i,\rperm(i)}^2} < \sqrt{\med{\sum_{i=1}^n a_{i,\rperm(i)}^2}} + v\max_{1\leq i,j \leq n}\ac{a_{i,j}}.
\end{equation}
Then, as in the preliminary study of the proof of Lemma \ref{lm:concRac}, from the first assumption in \eqref{eq:hypPos}, there exists some $s^1,\dots,s^m$ in $U_A(\rperm)$, and some non-negative weights $p_1,\dots,p_m$ satisfying $\sum_{j=1}^m p_j=1$ such that
$$\sum_{i=1}^n \cro{\pa{\sum_{j=1}^m p_j s_i^j}^2} < t^2.$$
For each $1\leq j\leq m$, consider $\tau_j$ in $A$ associated to $s^j$, that is a permutation $\tau_j$ in $A$ satisfying $s_i^j=0 \implies \tau_j(i)=\rperm(i)$. 
Then, combining the Cauchy-Shwartz inequality with the second assumption in \eqref{eq:hypPos} leads to
\begin{eqnarray*}
Z-\sum_{j=1}^m p_j Z(\tau_j) &=& \sum_{i=1}^n \sum_{j=1}^m p_j \pa{a_{i,\rperm(i)}-a_{i,\tau_j(i)}}s_i^j\\ 
&\leq & \sum_{i=1}^n \cro{\pa{\sum_{j=1}^m p_j s_i^j} a_{i,\rperm(i)}} \\
&\leq & \sqrt{\sum_{i=1}^n \pa{\sum_{j=1}^m p_j s_i^j}^2} \sqrt{\sum_{i=1}^n a_{i,\rperm(i)}^2}\\
&< & t \pa{\sqrt{\med{\sum_{i=1}^n a_{i,\rperm(i)}^2}} + v\max_{1\leq i,j \leq n}\ac{a_{i,j}}}.
\end{eqnarray*}
Notice that here, the reasoning begins exactly as in the proof of Lemma \ref{lm:concRac}. Yet, the second assumption in \eqref{eq:hypPos}, which can be controlled thanks to that lemma, allows us to sharpen the inequality.
Thus, as the $\tau_j$ are in $A=\ac{\tau ; Z(\tau)\leq C_A}$, 
\begin{equation}
\label{eq:csqPos} 
Z < C_A + t \pa{\sqrt{\med{\sum_{i=1}^n a_{i,\rperm(i)}^2}} + v\max_{1\leq i,j \leq n}\ac{a_{i,j}}}. 
\end{equation}
Hence, by contraposition of \eqref{eq:hypPos} $\implies$ \eqref{eq:csqPos},  one obtains
\begin{multline*}
\proba{Z \geq C_A + t \pa{\sqrt{\med{\sum_{i=1}^n a_{i,\rperm(i)}^2}} + v\max_{1\leq i,j \leq n}\ac{a_{i,j}}}} \\
\leq \proba{f(A,\rperm)\geq t^2} + \proba{\sqrt{\sum_{i=1}^n a_{i,\rperm(i)}^2} \geq \sqrt{\med{\sum_{i=1}^n a_{i,\rperm(i)}^2}} + v\max_{1\leq i,j \leq n}\ac{a_{i,j}}}, 
\end{multline*}
and \eqref{eq:1stepPos} follows from Theorem \ref{thm:Talagrand} and \eqref{eq:concRac+} in Lemma \ref{lm:concRac}. 

\paragraph{Second step: proof of \eqref{eq:concPosMed+}.}
Consider $C_A=\med{Z}$ so that $\proba{Z\in A}\geq 1/2$.
Thus, if $v=t$ in \eqref{eq:1stepPos}
\begin{align*}
\PP\Bigg(Z \geq \med{Z} &+ t\pa{\sqrt{\med{\sum_{i=1}^n a_{i,\rperm(i)}^2}} + (t\vee C_0) \max_{1\leq i,j\leq n}
\ac{a_{i,j}}}\Bigg)\\
&\leq \proba{Z \geq \med{Z} + t\pa{\sqrt{\med{\sum_{i=1}^n a_{i,\rperm(i)}^2}} + t \max_{1\leq i,j\leq n}
\ac{a_{i,j}}}}\\
&\leq 4 e^{-t^2/16}
\end{align*}
Notice that the maximum with the constant in $(t\vee C_0)$ is not necessary in the case only a control of the right-tail is wanted. 

\paragraph{Third step: proof of \eqref{eq:concPosMed-}.}
Consider now $$C_A = \med{Z} - t \pa{\sqrt{\med{\sum_{i=1}^n a_{i,\rperm(i)}^2}} + v\max_{1\leq i,j \leq n}\ac{a_{i,j}}},$$
so that
$$\proba{Z \geq C_A + t \pa{\sqrt{\med{\sum_{i=1}^n a_{i,\rperm(i)}^2}} + v\max_{1\leq i,j \leq n}\ac{a_{i,j}}}}
= \proba{Z\geq \med{Z}} \geq \frac{1}{2}.$$
Hence, on the one hand, from \eqref{eq:1stepPos}, 
$$\proba{Z\in A} \leq \frac{e^{-t^2/16}}{\pa{\frac{1}{2} - 2 e^{-v^2/16}}}.$$ 
Thus, if $v=C_0=4\sqrt{\ln(8)}$, then $\pa{\frac{1}{2} - 2 e^{-v^2/16}} = \frac{1}{4}$, and $\proba{Z\in A} \leq 4 e^{-t^2/16}.$

On the other hand, as $(t\vee C_0) \geq C_0=v$, 
$$\proba{Z\in A} \geq \proba{Z \leq \med{Z} - t\pa{\sqrt{\med{\sum_{i=1}^n a_{i,\rperm(i)}^2}} + (t\vee C_0) \max_{1\leq i,j\leq n}\ac{a_{i,j}}}},$$ 
which ends the proof of \eqref{eq:concPosMed-}.

\paragraph{Fourth step: proof of \eqref{eq:concPosMed}.}
Both \eqref{eq:concPosMed+} and \eqref{eq:concPosMed-} lead to 
$$\proba{\abs{Z - \med{Z}} > t\pa{\sqrt{\med{\sum_{i=1}^n a_{i,\rperm(i)}^2}} + (t\vee C_0) \max_{1\leq i,j\leq n}\ac{a_{i,j}}}}\leq 8 e^{-t^2/16}.$$
Thus, on the one hand, if $t\geq C_0$, that is $t\vee C_0 = t$, and \eqref{eq:concPosMed} holds. \\
On the other hand, if $t < C_0$,
\begin{multline*}
\proba{\abs{Z - \med{Z}} > t\pa{\sqrt{\med{\sum_{i=1}^n a_{i,\rperm(i)}^2}} + t\max_{1\leq i,j\leq n}\ac{a_{i,j}}}}
\leq  1\\
\leq  e^{C_0^2/16 - t^2/16} = 8 e^{-t^2/16},
\end{multline*}
which ends the proof of the Proposition by taking $x=\sqrt{t}$.

\subsection{Proof of Lemma \ref{lm:ineqEspMedVar}}
\label{sct:lmineqEspMedVar}

Let $X$ be any real random variable.
Recall that $$\med{X} \in \argmin_{m\in\R} \esp{\abs{X-m}}. $$
In particular, thanks to Jensen's inequality, 
\begin{eqnarray}
\abs{\esp{X}-\med{X}} &\leq & \esp{\abs{X-\med{X}}} \nonumber \\
&\leq & \esp{\abs{X-\esp{X}}} \nonumber\\
&\leq & \sqrt{\esp{\pa{X-\esp{X}}^2}} \nonumber\\
&\leq & \sqrt{\var{X}} \label{eq:majdiffmoymed}.
\end{eqnarray}

\subsection{Proof of Proposition \ref{prop:concPosMoy}}
\label{sct:propConcPosMoy}

First, for a better readability, let 
$$M=\max_{1\leq i,j\leq n} \ac{a_{i,j}}\quad \mbox{ and }\quad V=\esp{\sum_{i=1}^n a_{i,\rperm(i)}^2}=\frac{1}{n}\sum_{i,j=1}^n a_{i,j}^2.$$
Then, $\med{\sum_{i=1}^n a_{i,\rperm(i)}^2} \leq 2V$ since by Markov's inequality, for all non-negative random variable X, $\med{X}\leq 2\esp{X}$. 
Indeed, 
$$\frac{1}{2} \leq \proba{X \geq \med{X}} \leq \frac{\esp{X}}{\med{X}}. $$

Thus, by Proposition \ref{prop:concPosMed}, one obtains that, for all $x>0$,
\begin{equation}
\label{eq:tradconcMedMV}
\proba{\abs{Z-\med{Z}}\geq \sqrt{2Vx} + Mx } \leq 8 e^{-x/16}.
\end{equation}

The following is based on Lemma \ref{lm:ineqEspMedVar}, and provides an upper-bound of the difference between the expectation and the median of $Z$. 

\begin{lm}
\label{lm:majdiffmoymed}
With the notation defined above, 
$$
\abs{\esp{Z}-\med{Z}} \leq \sqrt{2 V}.
$$
\end{lm}

\begin{proof}[Proof of Lemma \ref{lm:majdiffmoymed}] 
Lemma \ref{lm:ineqEspMedVar} implies that 
$$
\abs{\esp{Z}-\med{Z}} \leq \sqrt{\var{Z}}.
$$
Let us prove that  
\begin{equation}
\label{eq:majvar}
\var{Z} \leq 2V. 
\end{equation}

Indeed, 
\begin{eqnarray*}
\var{Z} &=& \esp{\pa{\sum_{i=1}^n a_{i,\rperm(i)} - \frac{1}{n}\sum_{i,j=1}^n a_{i,j}}^2} \\
&=& \esp{\pa{\sum_{i,j=1}^n a_{i,j}\pa{\1{\rperm(i)=j}-\frac{1}{n}}}^2} \\
&=& \sum_{i,j=1}^n \sum_{k,l=1}^n a_{i,j} a_{k,l} E_{i,j,k,l},
\end{eqnarray*}
where $$E_{i,j,k,l}=\esp{\pa{\1{\rperm(i)=j}-\frac{1}{n}}\pa{\1{\rperm(k)=l}-\frac{1}{n}}} = \esp{\1{\rperm(i)=j}\1{\rperm(k)=l}} - \frac{1}{n^2}.$$
In particular, 
\begin{equation*}
E_{i,j,k,l} = \left\{ \begin{array}{ll} 
\displaystyle\frac{1}{n}-\frac{1}{n^2} \leq \frac{1}{n} & \mbox{if } i=k \mbox{ and } j=l, \\ \\
\displaystyle\frac{-1}{n^2}\leq 0 & \mbox{if } i=k \mbox{ and } j\neq l \mbox{ or } i\neq k \mbox{ and } j=l, \\ \\
\displaystyle\frac{1}{n(n-1)} - \frac{1}{n^2} = \frac{1}{n^2(n-1)} & \mbox{if } i\neq k \mbox{ and } j\neq l. 
\end{array}\right. 
\end{equation*}
Therefore, from the Cauchy-Schwarz inequality applied to the second sum bellow (of $n^2(n-1)^2$ terms), one obtains
\begin{eqnarray*}
\var{Z} &\leq & \frac{1}{n} \sum_{i,j=1}^n a_{i,j}^2 + \frac{1}{n^2(n-1)} \sum_{i\neq k} \sum_{j\neq l} a_{i,j} a_{k,l} \\
&\leq & V + \frac{\sqrt{n^2(n-1)^2}}{n^2(n-1)} \sqrt{\sum_{i\neq k}\sum_{j\neq l} a_{i,j}^2 a_{k,l}^2} \\
&\leq & V + \frac{1}{n} \sqrt{\sum_{i,j} a_{i,j}^2\sum_{k,l}a_{k,l}^2} \\
&=& 2V.
\end{eqnarray*}
Finally, combining \eqref{eq:majdiffmoymed} and \eqref{eq:majvar} ends the proof of Lemma \ref{lm:majdiffmoymed}. 
\end{proof}


Therefore, one deduces from Lemma \ref{lm:majdiffmoymed} and Equation \eqref{eq:tradconcMedMV} that for all $x>0$, 
\begin{equation}
\label{eq:concMoyPlusEta}
\proba{\abs{Z-\esp{Z}} \geq \sqrt{2 V} + \sqrt{2 V x} + M x} \leq 8e^{-x/16}.
\end{equation}

Now, as in \cite[Corollary 2.11]{BoucheronLugosiMassart2013}, introduce $h_1 : u\in\R^+ \mapsto 1+u-\sqrt{1+2u}$. Then, in particular, $h_1$ is non-decreasing, convex, one to one function on $\R^+$ with inverse function $h_1^{-1} : v\in\R^+ \mapsto v+\sqrt{2v}.$ 
Indeed, 
\begin{eqnarray*}
h_1\pa{h_1^{-1}(v)} &=& 1+v+\sqrt{2v} - \sqrt{1+2v+2\sqrt{2v}} \\ 
&=& 1 + v +\sqrt{2v} - \sqrt{\pa{1+\sqrt{2v}}^2} =v,
\end{eqnarray*}
and 
\begin{eqnarray*}
h_1^{-1}\pa{h_1(u)} &=& 1+u-\sqrt{1+2u} + \sqrt{2 + 2u-2\sqrt{1+2u}} \\
&=& u + 1-\sqrt{1+2u} + \sqrt{1 -2\sqrt{1+2u} + 1+2u} \\
&=& 1+u-\sqrt{1+2u} + \sqrt{\pa{1-\sqrt{1+2u}}^2} = u. \\
\end{eqnarray*}

Consider $\ca$ and $\cc$ defined by $\ca = V/M$ and $\cc = M^2/V$, such that $\ca \cc = M $ and $\ca^2 \cc = V$ and thus
$$\sqrt{2 V x} + M x  = \ca h_1^{-1}(\cc x).$$

Then, from \eqref{eq:concMoyPlusEta}, 
$$\proba{\abs{Z-\esp{Z}}\geq \sqrt{2\ca^2 \cc} + \ca h_1^{-1}(\cc x)}\leq 8 e^{-x/16}.$$

\medskip
Let $t>0$, and consider the two following cases.

\paragraph{1st case:} if $t \geq \sqrt{2V} = \sqrt{2\ca^2 \cc}$, 
then define $\displaystyle x=\frac{1}{\cc}h_1\pa{\frac{t}{\ca}-\sqrt{2\cc}}$ such that $t = \sqrt{2\ca^2 \cc}+\ca h_1^{-1}(\cc x)$. Then, 
$$\proba{\abs{Z-\esp{Z}}\geq t} \leq 8 \exp\pa{-\frac{1}{16\cc}h_1\pa{\frac{t}{\ca}-\sqrt{2\cc}}}. $$
Yet, by convexity of $h_1$, 
$$h_1\pa{\frac{t}{\ca}-\sqrt{2\cc}} \geq 2 h_1\pa{\frac{t}{2\ca}} - h_1\pa{\sqrt{2\cc}}.$$
Hence,
$$\proba{\abs{Z-\esp{Z}}\geq t} \leq 8 \exp\pa{\frac{1}{16\cc}h_1\pa{\sqrt{2\cc}}}\exp\pa{-\frac{1}{8\cc}h_1\pa{\frac{t}{2\ca}}}.$$

Moreover, 
$\sqrt{2\cc} \leq \cc + \sqrt{2\cc} = h_1^{-1} \pa{\cc},$
hence $$\frac{1}{16\cc}h_1\pa{\sqrt{2\cc}}\leq \frac{1}{16}.$$
So finally in this case, 
\begin{equation}
\label{eq:grandst}
\proba{\abs{Z-\esp{Z}}\geq t} \leq 8 e^{1/16}\exp\pa{-\frac{1}{8\cc}h_1\pa{\frac{t}{2\ca}}}.
\end{equation}

\paragraph{2nd case:} if $t < \sqrt{2V} = \sqrt{2\ca^2 \cc}$, 
\begin{equation*}
\proba{\abs{Z-\esp{Z}}\geq t} \quad\leq\quad 1
\quad = \quad  \exp\pa{\frac{1}{8\cc}h_1\pa{\frac{t}{2\ca}}}\exp\pa{-\frac{1}{8\cc}h_1\pa{\frac{t}{2\ca}}}
\end{equation*} 

Moreover, in this case, since $\sqrt{2\cc}/2 \leq h_1^{-1}(\cc/4)$, hence
$$\frac{1}{8\cc} h_1\pa{\frac{t}{2\ca}} \leq \frac{1}{8\cc} h_1\pa{\frac{\sqrt{2\cc}}{2}}\leq \frac{1}{32},$$
and thus
\begin{equation}
\label{eq:petitst}
\proba{\abs{Z-\esp{Z}}\geq t} \leq  e^{1/32}\exp\pa{-\frac{1}{8\cc}h_1\pa{\frac{t}{2\ca}}}.
\end{equation} 

\medskip

\paragraph{}
Finally, combining \eqref{eq:grandst} and \eqref{eq:petitst} leads, in all cases, to 
\begin{equation}
\label{eq:toust}
\proba{\abs{Z-\esp{Z}}\geq t} \leq 8e^{1/16} \exp\pa{-\frac{1}{8\cc}h_1\pa{\frac{t}{2\ca}}}.
\end{equation}

Now, in order to obtain the Bernstein-type inequality, let $\displaystyle x = \frac{2}{\cc}h_1\!\pa{\frac{t}{2\ca}}$, then 
$$t = 2\ca h_1^{-1}\pa{\frac{\cc x}{2}} = \ca\cc x + 2\sqrt{\ca^2\cc x} = 2\sqrt{V x} + Mx,$$
and thus for all $x>0$, 
\begin{equation}
\proba{\abs{Z-\esp{Z}}\geq 2\sqrt{V x} + M x}\leq 8e^{1/16} \exp\pa{-\frac{x}{16}},
\end{equation}
which ends the proof of the Proposition.

\subsection{Proof of Corollary \ref{coro:concPosMoy}}
\label{sct:coroConcPosMoy}

Consider the same notation as in both Proposition \ref{prop:concPosMoy} and its proof.
This proof follows the one of \cite[Corollary 2.10]{Massart2007}. 
Notice that for all $u\geq 0$,
$$h_1(u)\geq \frac{u^2}{2(1+u)}.$$ 
Hence, from \eqref{eq:toust} in the proof of Proposition \ref{prop:concPosMoy}, for all $t \geq 0$, 
\begin{eqnarray*}
\proba{\abs{Z-\esp{Z}}\geq t} &\leq & 8e^{1/16} \exp\pa{-\frac{1}{8\cc}h_1\pa{\frac{t}{2\ca}}} \\
&\leq & 8e^{1/16} \exp\pa{-\frac{t^2}{64\ca^2\cc\pa{1+t/2\ca}}} \\
&=& 8e^{1/16} \exp\pa{-\frac{t^2}{32\pa{2\ca^2\cc+\ca\cc t}}} \\
&=& 8e^{1/16} \exp\pa{-\frac{t^2}{32\pa{V+M t}}}. 
\end{eqnarray*}
which ends the proof of the Corollary.

\subsection{Proof of Theorem \ref{thm:concQcqMoy}}
\label{sct:thmConcQcqMoy}

For a better readability, introduce  
$a_{i,j}^+ = a_{i,j}\1{a_{i,j}\geq 0}$ (respectively $a_{i,j}^- = -a_{i,j}\1{a_{i,j} < 0}$), and 
denote $Z^+ = \sum_{i=1}^n a_{i,\rperm(i)}^+$ (respectively $Z^- = \sum_{i=1}^n a_{i,\rperm(i)}^-$).
Then $$Z = \sum_{i=1}^n a_{i,\rperm(i)} = Z^+ - Z^-.$$

Moreover, if $v$ (respectively $v^+$ and $v^-$) denotes $\frac{1}{n}\sum_{i,j=1}^n a_{i,j}^2$ (respectively $\frac{1}{n}\sum_{i,j=1}^n (a_{i,j}^+)^2$ and $\frac{1}{n}\sum_{i,j=1}^n (a_{i,j}^-)^2$), then $v=v^+ + v^-$ and, from the concavity property of the square root function, 
$$\sqrt{2v}\geq \sqrt{v^+} + \sqrt{v^-}.$$

Furthermore, if $M^+$ (respectively $M^-$) denotes $\max_{1\leq i,j\leq n}\{a_{i,j}^+\}$ (respectively $\max_{1\leq i,j\leq n}\{a_{i,j}^-\}$), then $2M=2\max_{1\leq i,j\leq n}\ac{|a_{i,j}|} \geq M^+ + M^-$. 

Finally, applying Proposition \ref{prop:concPosMoy} to $Z^+$ and $Z^-$ which are both sums of non-negative numbers leads to
\begin{eqnarray*}
\PP\Big(|Z \!\! &-&\!\! \esp{Z}| \ \geq\ 2\sqrt{2vx}\ +\ 2Mx\Big) \\
&\leq & \proba{\abs{Z^+ -\esp{Z^+}} + \abs{Z^- -\esp{Z^-}} \geq 2\sqrt{v^+ x} + M^+ x + 2\sqrt{v^- x} + M^- x} \\
&\leq &\proba{\abs{Z^+ -\esp{Z^+}} \geq 2\sqrt{v^+ x} + M^+ x} +\ \proba{\abs{Z^- -\esp{Z^-}} \geq 2\sqrt{v^- x} + M^- x} \\
&\leq & 16e^{1/16} \exp\pa{-\frac{x}{16}},
\end{eqnarray*}
which ends the proof of the Theorem.

\subsection{Proof of Corollary \ref{coro:concQcqMoy}}
\label{sct:thmConcQcqMoy}

Consider the same notation as in the proof of Theorem \ref{thm:concQcqMoy}, and let $t>0$.  
Let $M$ denote the maximum $\max_{1\leq i,j\leq n}\ac{|a_{i,j}|}$. On the one hand, $M^+ \leq M$ and $M^- \leq M$, and on the other hand, $v^+\leq v$ and $v^- \leq v$. 
Therefore, applying Corollary \ref{coro:concPosMoy}, one obtains
\begin{eqnarray*}
\proba{\abs{Z-\esp{Z}}\geq t} 
&\leq & \proba{\abs{Z^+ -\esp{Z^+}} + \abs{Z^- -\esp{Z^-}}\geq t} \\
&\leq & \proba{\abs{Z^+ -\esp{Z^+}}\geq t/2} + \proba{\abs{Z^- -\esp{Z^-}}\geq t/2} \\
&\leq & 8 e^{1/16} \exp\pa{\frac{-(t/2)^2}{16\pa{4v^+ + 2M^+ t/2}}} + 8 e^{1/16}\exp\pa{\frac{-(t/2)^2}{16\pa{4v^- + 2M^- t/2}}} \\
&\leq & 16e^{1/16} \exp\pa{\frac{-t^2}{64\pa{4 v + M t}}},
\end{eqnarray*}
which leads to the following intermediate result
\begin{equation}
\label{eq:concQcqMoyInter}
\proba{\abs{Z-\esp{Z}}\geq t}\leq 16e^{1/16} \exp\pa{\frac{-t^2}{64\pa{4\frac{1}{n}\sum_{i,j=1}^n a_{i,j}^2 + \max_{1\leq i,j\leq n} \ac{\abs{a_{i,j}}} t}}}.
\end{equation}

In order to make the variance appear, consider Hoeffding's centering trick recalled in \eqref{eq:def_dij} and introduce 
$$d_{i,j} = a_{i,j} - \frac{1}{n} \sum_{k=1}^n a_{k,j} -\frac{1}{n} \sum_{l=1}^n a_{i,l} + \frac{1}{n^2} \sum_{k,l=1}^n a_{k,l}
=\frac{1}{n^2} \sum_{k,l=1}^n \pa{a_{i,j}-a_{k,j}-a_{i,l}+a_{k,l}}.$$
One may easily verify that for all $i_0$ and $j_0$, $\sum_{i=1}^n d_{i,j_0} = \sum_{j=1}^n d_{i_0,j} = 0$.
Moreover, 
$$\sum_{i=1}^n d_{i,\rperm(i)} = \sum_{i=1}^n a_{i,\rperm(i)} - \frac{1}{n} \sum_{i,j=1}^n a_{i,j}=Z-\esp{Z} \quad \mbox{and}\quad \esp{\sum_{i=1}^n d_{i,\rperm(i)}} = \frac{1}{n}\sum_{i,j=1}^n d_{i,j} = 0.$$
In particular, applying equation \eqref{eq:concQcqMoyInter} to the permuted sum of the $d_{i,j}$'s leads to 
\begin{equation}
\label{eq:concQcqMoyHoeffding}
\proba{\abs{Z-\esp{Z}}\geq t}\leq 16e^{1/16} \exp\pa{\frac{-t^2}{64\pa{4\frac{1}{n}\sum_{i,j=1}^n d_{i,j}^2 + \max_{1\leq i,j\leq n} \ac{\abs{d_{i,j}}} t}}}.
\end{equation}
Then, it is sufficient to notice that, on the one hand,  from \cite[Theorem 2]{Hoeffding1951}, 
$$\var{Z}=(n-1)^{-1}\sum_{i,j=1}^n d_{i,j}^2\geq n^{-1}\sum_{i,j=1}^n d_{i,j}^2,$$ and on the other hand, $$\max_{1\leq i,j\leq n} \ac{\abs{d_{i,j}}} \leq 4 \max_{1\leq i,j\leq n} \ac{\abs{a_{i,j}}}, $$
to end the proof of Corollary \ref{coro:concQcqMoy}.

\subsection{Proof of Proposition~\ref{prop:control_quantile}}
\label{sect:proof_prop_control_var_quantile}

The proof of Proof of Proposition~\ref{prop:control_quantile} is divided into two steps. The first step consists in controlling the conditional quantile $q_{1-\alpha}(\X_n)$ and the second step provides an upper-bound for $q_{1-\beta/2}^\alpha$. 

\begin{description}
\item[1st step.] Let us prove \eqref{eq:control_conditional_quantile}, that is
$$
q_{1-\alpha}(\X_n) \leq \frac{C'}{n-1} \ac{\sqrt{\frac{1}{n} \sum_{i,j=1}^n\varphi_\delta^2(X_i^1,X_j^2)} \sqrt{\ln\pa{\frac{c_0}{\alpha}}} + \norm{\varphi_\delta}_{\infty}\ln\pa{\frac{c_0}{\alpha}}}. 
$$
Introduce $\tilde Z(\X_n) = \sum_{i=1}^n \varphi_\delta (X_i^1,X_{\rperm(i)}^2)$. Then, notice that 
\begin{equation}
\label{eq:writing_stat}
T_\delta^\rperm (\X_n) = \frac{1}{n-1}\pa{\tilde Z(\X_n) - \esp{\tilde Z(\X_n)\middle| \X_n}}.
\end{equation}
Therefore, applying Theorem~\ref{thm:concQcqMoy} to the conditional probability given $\X_n$, one obtains that there exist universal positive constants $c_0$ and $c_1$ such that, for all $x>0$, 
\begin{multline*}
\proba{\abs{\tilde Z(\X_n)-\esp{\tilde Z(\X_n)\middle| \X_n}} \geq  2\sqrt{2\pa{\frac{1}{n}\sum_{i,j=1}^n \varphi_\delta^2(X_i^1,X_j^2)}x} + 2\norm{\varphi_\delta}_\infty x\middle|\X_n} \\
\leq c_0 \exp\pa{-c_1 x}.
\end{multline*}
In particular, from \eqref{eq:writing_stat}, one obtains 
\begin{multline*}
\proba{\abs{T_\delta(\X_n^{\rperm})} \geq  \frac{2}{n-1}\pa{\sqrt{2\pa{\frac{1}{n}\sum_{i,j=1}^n \varphi_\delta^2(X_i^1,X_j^2)}x} + \norm{\varphi_\delta}_\infty x}\middle| \X_n} \\
\leq c_0 \exp\pa{- c_1 x}.
\end{multline*}

Yet, by definition of the quantile, $q_{1-\alpha}(\X_n)$ is the smallest $u$ such that 
$$\proba{\abs{T_\delta(\X_n^{\rperm})} \geq u\middle|\X_n} \leq \alpha.$$ 
Thus taking $x$ such that $c_0 \exp\pa{-c_1 x}=\alpha$ , that is $x=c_1^{-1} \ln\pa{c_0/\alpha}$, one obtains \eqref{eq:control_conditional_quantile} with $C' = 2 \max\ac{\sqrt{2/c_1}, 1/c_1}$ which is a universal positive constant. 

\item[2nd step.] Let us now control the quantile $q_{1-\beta/2}^{\alpha}$.
Since \eqref{eq:control_conditional_quantile} is always true, by definition of $q_{1-\beta/2}^\alpha$, one has that $q_{1-\beta/2}^\alpha$ is upper bounded by the $(1-\beta/2)$-quantile of the right-hand side of \eqref{eq:control_conditional_quantile}.
Yet, the only randomness left in the right-hand side of \eqref{eq:control_conditional_quantile} comes from the randomness of $\frac{1}{n}\sum_{i,j=1}^n \varphi_\delta^2(X_i^1,X_j^2)$, and thus it is sufficient to control its $(1-\beta/2)$-quantile.

Besides, applying Markov's inequality, one obtains for all $x>0$, 
$$\proba{\frac{1}{n}\sum_{i,j=1}^n \varphi_\delta^2(X_i^1,X_j^2) \geq x} \leq \frac{\esp{\frac{1}{n}\sum_{i,j=1}^n \varphi_\delta^2(X_i^1,X_j^2) }}{x}, $$
with $\esp{\frac{1}{n}\sum_{i,j=1}^n \varphi_\delta^2(X_i^1,X_j^2) } = \esps{P}{\varphi_\delta^2} + (n-1) \esps{\indep}{\varphi_\delta^2}$, and thus, taking $$x=\frac{2}{\beta}\pa{\esps{P}{\varphi_\delta^2} + (n-1) \esps{\indep}{\varphi_\delta^2}},$$
one has that the $(1-\beta/2)$-quantile of $\frac{1}{n}\sum_{i,j=1}^n \varphi_\delta^2(X_i^1,X_j^2)$ is upper bounded by $x$, and thus, the  $(1-\beta/2)$-quantile of $\sqrt{\frac{1}{n}\sum_{i,j=1}^n \varphi_\delta^2(X_i^1,X_j^2)}$ is itself upper bounded by 
$$\sqrt{\frac{2}{\beta}}\pa{\sqrt{\esps{P}{\varphi_\delta^2}} + \sqrt{n}\sqrt{\esps{\indep}{\varphi_\delta^2}}}.$$
Finally, 
\begin{eqnarray*}
q_{1-\beta/2}^\alpha &\leq & \frac{2C'}{n} \ac{\sqrt{\frac{2}{\beta}}\pa{\sqrt{\esps{P}{\varphi_\delta^2}} + \sqrt{n}\sqrt{\esps{\indep}{\varphi_\delta^2}}} \sqrt{\ln\pa{\frac{c_0}{\alpha}}} + \norm{\varphi_\delta}_{\infty}\ln\pa{\frac{c_0}{\alpha}}}.
\end{eqnarray*}
which is exactly \eqref{eq:control_quantile} for any constant $C\geq 2C'$.
\end{description}

\section{Proof of Lemma~\ref{lm:calcul_variance}}
\label{sct:proof_lm_calcul_variance}

Let us now prove Lemma \ref{lm:calcul_variance}. 
Let $n\geq 4$ and $\X_n$ be an i.i.d. sample with distribution $P$. 
First notice that one can write
$$T_\delta(\X_n) = \frac{1}{n(n-1)} \sum_{i\neq j} \pa{\varphi_\delta (X_i^1,X_i^2) - \varphi_\delta(X_i^1,X_j^2)}.$$
In particular, one recovers that $\esp{T_\delta(\X_n)} = \esps{P}{\varphi_\delta}-\esps{\indep}{\varphi_\delta}$. 

For a better readability, let us introduce for all $i\neq j$ in \ac{1,2,\ldots,n}, 
$$Y_i = \varphi_\delta(X_i^1,X_i^2)-\esps{P}{\varphi_\delta} \quad \mbox{and}\quad Z_{i,j} = \varphi_\delta(X_i^1,X_j^2)-\esps{\indep}{\varphi_\delta}.$$ Then, 
\begin{equation}
\label{eq:propYZ}
\esp{Y_i} = \esp{Z_{i,j}} = 0,\quad \mbox{and} \quad 
\left\{\begin{array}{l}
\displaystyle \esp{Y_i^2} = \vars{P}{\varphi_\delta}\leq \esps{P}{\varphi_\delta^2},\\ 
\displaystyle \esp{Z_{i,j}^2} = \vars{\indep}{\varphi_\delta}\leq \esps{\indep}{\varphi_\delta^2}.
\end{array}\right.
\end{equation}
 
One can write $$T_\delta(\X_n) - \esp{T_\delta(\X_n)} = \frac{1}{n(n-1)} \sum_{i\neq j} \pa{Y_i - Z_{i,j}},$$ and thus, 
\begin{eqnarray*}
\var{T_\delta(\X_n)} &=& \esp{\pa{\frac{1}{n(n-1)} \sum_{i\neq j} \pa{Y_i - Z_{i,j}}}^2} \\ 
&=& \frac{1}{n^2(n-1)^2} \sum_{i\neq j} \sum_{k\neq l} \esp{\pa{Y_i-Z_{i,j}}\pa{Y_k - Z_{k,l}}} \\
&=& A_n - 2 B_n + C_n, \\
\end{eqnarray*}
with 
$$A_n = \frac{1}{n^2} \sum_{i,k=1}^n \esp{Y_i Y_k},$$ 
$$B_n = \frac{1}{n^2(n-1)} \sum_{i=1}^n\sum_{k\neq l} \esp{Y_i Z_{k,l}},$$
$$C_n = \frac{1}{n^2(n-1)^2} \sum_{i\neq j} \sum_{k\neq l} \esp{Z_{i,j} Z_{k,l}},$$
where each sum is taken for indexes contained in $\ac{1,2,\ldots,n}$. 
In particular, since just an upper-bound of the variance is needed, it is sufficient to write 
\begin{equation}
\label{eq:control_variance1}
\var{T_\delta(\X_n)} \leq |A_n| + 2 |B_n| + |C_n|, 
\end{equation}
and to study each term separately.

\paragraph{Study of $A_n$.} Since by construction, the $Y_i$'s are centered, and independent (as the $X_i$'s are), 
\begin{eqnarray*}
A_n &=& \frac{1}{n^2} \pa{ \sum_{i} \esp{Y_i^2} + \sum_{i\neq k} \esp{Y_i}\esp{Y_k}} \\ 
&=& \frac{1}{n} \esp{Y_1^2}, \\
\end{eqnarray*}
and in particular, from \eqref{eq:propYZ}, 
\begin{equation}
\label{eq:control_A}
\abs{A_n} \leq \frac{1}{n} \esps{P}{\varphi_\delta^2}. 
\end{equation}

\paragraph{Study of $B_n$.} If $i$, $k$ and $l$ are all different, using once again the independence of the $X_i$'s and a centering argument, then $\esp{Y_{i}Z_{k,l}} = \esp{Y_{i}}\esp{Z_{k,l}} = 0$. Thus 
\begin{eqnarray*}
B_n &=& \frac{1}{n^2 (n-1)} \sum_{i\neq k} \pa{ \esp{Y_i Z_{i,k}} + \esp{Y_i Z_{k,i}} } \\ 
&=& \frac{1}{n} \pa{\esp{Y_1 Z_{1,2}} + \esp{Y_1 Z_{2,1}}}.
\end{eqnarray*}
In particular, applying the Cauchy-Schwartz inequality, and from \eqref{eq:propYZ}, one obtains
\begin{equation}
\label{eq:control_B}
\abs{B_n} \leq  \frac{2}{n}\sqrt{\esp{Y_1^2} \esp{Z_{1,2}^2}}
\leq  \frac{2}{n}\sqrt{\esps{P}{\varphi_\delta^2} \esps{\indep}{\varphi_\delta^2}}.
\end{equation}

\paragraph{Study of $C_n$.} Still by an independence and a centering argument, if $i$, $j$, $k$ and $l$ are all different, 
$\esp{Z_{i,j} Z_{k,l}} = \esp{Z_{i,j}}\esp{Z_{k,l}} = 0$. Thus, if $I_n^{[3]}$ denotes the set of triplets $(i,j,k)$ in $\ac{1,\dots,n}^3$ which are all different, one obtains
\begin{eqnarray*}
C_n &=& \frac{1}{n^2 (n-1)^2} \Bigg\{ \sum_{(i,j,k) \in I_n^{[3]}} \Big( \esp{Z_{i,j} Z_{i,k}} + 2 \esp{Z_{i,j} Z_{k,i}} +  \esp{Z_{j,i} Z_{k,i}} \Big) \\ 
&& \quad\quad\quad\quad\quad\quad +\ \sum_{i\neq j} \Big( \esp{Z_{i,j}^2} + \esp{Z_{i,j} Z_{j,i} }\Big) \Bigg\} \\
&=& \frac{n-2}{n(n-1)} \pa{\esp{Z_{1,2} Z_{1,3}} + 2 \esp{Z_{1,2} Z_{3,1}} + \esp{Z_{2,1} Z_{3,1}} }  \\ 
&& +\ \frac{1}{n(n-1)}\pa{\esp{Z_{1,2}^2} + \esp{Z_{1,2} Z_{2,1} }}.
\end{eqnarray*}
In particular, applying the Cauchy-Schwartz inequality, and using \eqref{eq:propYZ}, each expectation in the previous equation satisfies 
$\esp{Z_{i,j} Z_{k,l}}\leq \esp{Z_{1,2}^2} \leq \esps{\indep}{\varphi_\delta^2}$, and thus

\begin{equation}
\label{eq:control_C}
\abs{C_n} \leq  \pa{\frac{4(n-2)}{n(n-1)} + \frac{2}{n(n-1)}}\esps{\indep}{\varphi_\delta^2}
\leq  \frac{4}{n} \esps{\indep}{\varphi_\delta^2}.
\end{equation}

Finally, combining \eqref{eq:control_variance1}, \eqref{eq:control_A}, \eqref{eq:control_B}, and \eqref{eq:control_C} leads to 
$$\var{T_\delta(\X_n)} \leq \frac{1}{n}\pa{\sqrt{\esps{P}{\varphi_\delta^2}} + 2 \sqrt{\esps{\indep}{\varphi_\delta^2}}}^2, $$
which ends the proof of the Lemma.

%
%

\appendix
\section{A non-asymptotic control of the second kind error rates}
\label{sct:cond_Chebychev_Hoeffding_gen}

Consider the notation from Section \ref{sct:applitestindep}. Since this section focuses on the study of the second kind error rate of the test, in all the sequel, the observation is assumed to satisfy the alternative $(\hyp_1)$. 
Let thus $P$ be an alternative, that is $P\neq P^1\otimes P^2$, $n\geq 4$ and $\X_n=(X_i,\dots,X_n)$ be an i.i.d. sample from distribution $P$. Fix $\alpha$ and $\beta$ be two fixed values in $(0,1)$. Consider $T_\delta$ the test statistic introduced in \eqref{eq:def_stat_T}, the (random) critical value $q_{1-\alpha}(\X_n)$ defined in \eqref{eq:def_crit_valq}, and the corresponding permutation test defined in \eqref{eq:def_test_chap_conc} by $$\Delta_\alpha(\X_n) = \1{T_\delta(\X_n) > q_{1-\alpha}(\X_n)},$$
which precisely rejects independence when $T_\delta(\X_n) > q_{1-\alpha}(\X_n)$. 
Notice that this test is exactly the upper-tailed test by permutation introduced in \cite{AlbertBouretFromontReynaud2015}. 

\paragraph{}
The aim of this section is to provide different conditions on the alternative $P$ ensuring a control of the second kind error rate by a fixed value $\beta>0$, that is $\proba{\Delta_\alpha(\X_n)=0} \leq \beta.$ 
The following steps constitute the first steps of a general study of the separation rates for the previous independence test, and is worked through in the specific case of continuous real-valued random variables in \cite[Chapter 4]{Albert2015PhD}. 

Recall the notation introduced in \eqref{eq:not_esps_general_indep} for a better readability. For all real-valued measurable function $g$ on $\calX^2$, denote respectively 
\begin{equation*}
\esps{P}{g} = \esp{g(X_1^1,X_1^2)}\quad \mbox{and}\quad \esps{\indep}{g} = \esp{g(X_1^1,X_2^2)},
\end{equation*} 
the expectations of $g(X)$ under the alternative $P$ (meaning that $X \sim P$) and under the null hypothesis $(\hyp_0)$ (meaning that $X \sim P^1\otimes P^2$). \\

Assume the following moment assumption holds, that is
\medskip
\\
$\pa{\mc{A}_{Mmt,2}}\quad \textrm{\begin{tabular}{|l} both $\esps{P}{\varphi_\delta^2}<+\infty$ and $\esps{\indep}{\varphi_\delta^2}<+\infty$, \end{tabular}}$ 
\medskip \\
so that all variance and second-order moments exist. 
Then, the following statements hold. 
\begin{enumerate}
\item By Chebychev's inequality, one has $\proba{\Delta_\alpha(\X_n)=0} \leq \beta$  as soon as Condition \eqref{eq:cond_simple_tchebychev} is satisfied, that is 
$$\esp{T_\delta(\X_n)} \geq q^\alpha_{1-\beta/2} + \sqrt{\frac{2}{\beta} \var{T_\delta(\X_n)}}.$$

\item On the one hand,
\begin{equation}
\label{eq:control_variance_gen_gros}
\var{T_\delta(\X_n)} \leq \frac{8}{n} \pa{\esps{P}{\varphi_\delta^2} + \esps{\indep}{\varphi_\delta^2}}, 
\end{equation}

\item On the other hand, in order to control the quantile $q^\alpha_{1-\beta/2}$, let us first upper bound the conditional quantile, following Hoeffding's approach based on the Cauchy-Schwarz inequality, by
\begin{equation}
\label{eq:maj_quantile_Hoeffding}
q_{1-\alpha}(\X_n) \leq \sqrt{\frac{1-\alpha}{\alpha} \var{T_\delta\pa{\X_n^{\rperm}}\middle| \X_n}}.
\end{equation}

\item Markov's inequality allows us to deduce the following bound for the quantile: 
\begin{equation}
\label{eq:maj_quantile_du_quantile}
q^\alpha_{1-\beta/2} \leq 2\sqrt{\frac{1-\alpha}{\alpha}}\sqrt{\frac{2}{\beta}\frac{\pa{\esps{\indep}{\varphi_\delta^2} + \esps{P}{\varphi_\delta^2}}}{n}}.
\end{equation}

\item Finally, combining \eqref{eq:cond_simple_tchebychev}, \eqref{eq:control_variance_gen_gros} and \eqref{eq:maj_quantile_du_quantile} ensures that  $\proba{\Delta_\alpha(\X_n)=0} \leq \beta$ as soon as Condition \eqref{eq:cond_Hoeffding} is satisfied, that is 
$$\esp{T_\delta(\X_n)} \geq \frac{4}{\sqrt{\alpha}}\sqrt{\frac{2}{\beta}\frac{\esp{\varphi_\delta(X_1^1,X_1^2)^2} + \esp{\varphi_\delta(X_1^1,X_2^2)^2}}{n}}.$$

\end{enumerate}


\paragraph{}
This section is divided in five subsections, each one of them respectively proving a point stated above. The first one proves the sufficiency of Condition \eqref{eq:cond_simple_tchebychev} in order to control the second kind error rate. The second, third and fourth ones provide respectively upper-bounds of the variance term, the critical value and the quantile $q^\alpha_{1-\beta/2}$. Finally, the fifth one provides the sufficiency of Condition \eqref{eq:cond_Hoeffding}.

\subsection{A first condition ensuing from Chebychev's inequality}
\label{sct:cond_Chebychev}

In this section, we prove the sufficiency of a first simple condition, derived from Chebychev's inequality in order to control the second error rate. 
Assume that \eqref{eq:cond_simple_tchebychev} is satisfied, that is 
$$\esp{T_\delta(\X_n)} \geq q^\alpha_{1-\beta/2} + \sqrt{\frac{2}{\beta} \var{T_\delta(\X_n)}}.$$
Then, 
\begin{eqnarray}
\proba{\Delta_\alpha(\X_n)=0} &=& \proba{T_\delta(\X_n) \leq q_{1-\alpha}(\X_n)} \\
&=& \proba{\ac{T_\delta(\X_n) \leq q_{1-\alpha}(\X_n)} \cap \ac{q_{1-\alpha}(\X_n) \leq q^\alpha_{1-\beta/2}}} \nonumber \\
&& + \ \proba{\ac{T_\delta(\X_n) \leq q_{1-\alpha}(\X_n)} \cap \ac{q_{1-\alpha}(\X_n) > q^\alpha_{1-\beta/2}}} \nonumber \\
&\leq & \proba{T_\delta(\X_n) \leq q^\alpha_{1-\beta/2}} + \proba{q_{1-\alpha}(\X_n) > q^\alpha_{1-\beta/2}} \nonumber \\
&\leq & \proba{T_\delta(\X_n) \leq q^\alpha_{1-\beta/2}} + \frac{\beta}{2}, \label{eq:cheby1}
\end{eqnarray}
by definition of the quantile $q^\alpha_{1-\beta/2}$.
Yet, from \eqref{eq:cond_simple_tchebychev} one obtains from Chebychev's inequality that
\begin{eqnarray}
\proba{T_\delta(\X_n) \leq q^\alpha_{1-\beta/2}} &\leq & \proba{T_\delta(\X_n) \leq \esp{T_\delta(\X_n)} - \sqrt{\frac{2}{\beta} \var{T_\delta(\X_n)}}} \nonumber \\
&\leq & \proba{\abs{T_\delta(\X_n) - \esp{T_\delta(\X_n)}} \geq \sqrt{\frac{2}{\beta} \var{T_\delta(\X_n)}}} \nonumber \\
&\leq & \frac{\beta}{2}. \label{eq:cheby2}
\end{eqnarray}

Finally, both \eqref{eq:cheby1} and \eqref{eq:cheby2} lead to the desired control $\proba{\Delta_\alpha(\X_n)=0}\leq \beta$ which ends the proof. 

\subsection{Control of the variance in the general case}
\label{sct:control_var_general}

To upper bound the variance term, we apply Lemma \ref{lm:calcul_variance} which directly implies that

$$ \var{T_\delta(\X_n)} \leq  \frac{2}{n} \pa{\esps{P}{\varphi_\delta^2} + 4 \esps{\indep}{\varphi_\delta^2}}, $$
which directly leads to \eqref{eq:control_variance_gen_gros}.

\subsection{Control of the critical value based on Hoeffding's approach}
\label{sct:control_quantile_cond}

This section is devoted to the proof the inequality \eqref{eq:maj_quantile_Hoeffding}, namely
$$q_{1-\alpha}(\X_n) \leq \sqrt{\frac{1-\alpha}{\alpha} \var{T_\delta\pa{\X_n^{\rperm}}\middle| \X_n}}.$$

The proof of this upper-bound follows Hoeffding's approach in \cite{Hoeffding1952}, and relies on a normalizing trick, and the Cauchy-Schwarz inequality.
From now on, for a better readability, denote respectively $\Estar{\cdot}$ and $\Vstar{\cdot}$ the conditional expectation and variance given the sample $\X_n$. 

As in Hoeffding \cite{Hoeffding1952}, the first step is to center and normalize the permuted test statistic. 
Yet, by construction the permuted test statistic is automatically centered, that is $\Estar{T_\delta\pa{\X_n^{\rperm}}} = 0$, as one can notice that  
$$T_\delta\pa{\X_n^{\rperm}}=\frac{1}{n-1}\pa{\sum_{i=1}^n \varphi_\delta\pa{X_i^1,X_{\rperm(i)}^2} - \Estar{\sum_{i=1}^n \varphi_\delta\pa{X_i^1,X_{\rperm(i)}^2}}}. $$
Therefore, just consider the normalizing term
$$\nu_n(\ds X_n) = \Vstar{T_\delta\pa{\X_n^{\rperm}}} = \Estar{T_\delta\pa{\X_n^{\rperm}}^2}
= \frac{1}{n!}\sum_{\perm\in\Sn{n}} \pa{T_\delta\pa{\X_n^{\perm}}}^2.$$
Two cases appear: either $\nu_n(\ds X_n)=0$ or not. 

\paragraph{}
In the first case, the nullity of the conditional variance implies that all the permutations of the test statistic are equal.
Hence, for all permutation $\perm$ of $\ac{1,\dots,n}$, one has $T_\delta(\X_n^\perm) = T_\delta(\X_n)$.  Since the centering term $\Estar{\sum_{i=1}^n \varphi_\delta\pa{X_i^1,X_{\rperm(i)}^2}}=n^{-1}\sum_{i,j=1}^n\varphi_\delta(X_i^1,X_j^2)$ is permutation invariant, one obtains the equality of the permuted sums, that is 
$$\sum_{i=1}^n \varphi_\delta\pa{X_i^1,X_{\perm(i)}^2} = \sum_{i=1}^n \varphi_\delta\pa{X_i^1,X_i^2},$$ 
and this for all permutation $\perm$. 
In particular, the centering term is also equal to $\sum_{i=1}^n \varphi_\delta\pa{X_i^1,X_i^2}$. Indeed, by invariance of the sum (applied in the third equality below), 
\begin{eqnarray*}
\frac{1}{n} \sum_{i,j=1}^n \varphi_\delta\pa{X_i^1,X_j^2} &=& \frac{1}{n} \sum_{i,j=1}^n \varphi_\delta\pa{X_i^1,X_j^2}\cro{\frac{1}{(n-1)!}\sum_{\perm\in\Sn{n}} \1{\perm(i)=j}} \\
&=& \frac{1}{n!} \sum_{\perm\in\Sn{n}} \sum_{i=1}^n \varphi_\delta\pa{X_i^1,X_{\perm(i)}^2} \cro{\sum_{j=1}^n \1{\perm(i)=j}} \\
&=& \frac{1}{n!} \sum_{\perm\in\Sn{n}} \pa{\sum_{i=1}^n \varphi_\delta\pa{X_i^1,X_i^2}} \\
&=& \sum_{i=1}^n \varphi_\delta\pa{X_i^1,X_i^2}. 
\end{eqnarray*}
Therefore, $T_\delta(\X_n)$ is equal to zero, and thus, so is $q_{1-\alpha}(\X_n)$.
Finally, inequality \eqref{eq:maj_quantile_Hoeffding_norm} is satisfied since 
$$q_{1-\alpha}(\X_n) = 0 \leq 0 = \sqrt{\frac{1-\alpha}{\alpha} \var{T_\delta\pa{\X_n^{\rperm}}\middle| \X_n}}.$$

\paragraph{}
Consider now the second case, and assume $\nu_n\pa{\ds X_n}>0$. Let us introduce the (centered and) normalized statistic
$$T_\delta'\pa{\ds X_n} = \frac{1}{\sqrt{\nu_n(\ds X_n)}}\pa{T_\delta(\X_n)}.$$
In particular, the new statistic $T_\delta'\pa{\ds X_n}$ satisfies 
$$\Estar{T_\delta'\pa{\X_n^{\rperm}}} = 0 \quad \mbox{and}\quad \Vstar{T_\delta'\pa{\X_n^{\rperm}}}\leq 1.$$

One may moreover notice that the normalizing term $\nu_n(\X_n)$ is permutation invariant, that is, for all permutations $\perm$ and $\perm'$ in $\Sn{n}$, 
$$\nu_n\pa{\ds X_n^{\perm}}=\nu_n\pa{\ds X_n}=\nu_n\pa{\ds X_n^{\perm'}}. $$
In particular, since $\nu_n\pa{\ds X_n} >0$, 
$$T_\delta\pa{\X_n^{\perm}}\leq T_\delta\pa{\X_n^{\perm'}}\quad \Leftrightarrow \quad T_\delta'\pa{\X_n^{\perm}}\leq T_\delta'\pa{\X_n^{\perm'}}.$$ 

Therefore, as the test $\Delta_\alpha$ depends only on the comparison of the $\ac{T_\delta\pa{\X_n^{\perm}}}_{\perm\in\Sn{n}}$, the test statistic $T_\delta$ can be replaced by $T_\delta'$, and the new critical value becomes 
\begin{equation}
\label{eq:def_quantile_norm}
q'_{1-\alpha}(\X_n) = T_\delta^{\prime (n!-\lfloor n!\alpha\rfloor)}\pa{\ds X_n} = \frac{T_\delta^{(n!-\lfloor n!\alpha\rfloor)}\pa{\ds X_n}}{\nu_n(\X_n)}=\frac{q_{1-\alpha}(\X_n)}{\nu_n(\X_n)}. 
\end{equation}

Moreover, following the proof of Theorem 2.1. of Hoeffding \cite{Hoeffding1952}, one can show (as below) that 
\begin{equation}
\label{eq:maj_quantile_Hoeffding_norm}
q'_{1-\alpha}(\X_n) \leq \sqrt{\frac{1-\alpha}{\alpha}}.
\end{equation}
Hence, combining \eqref{eq:maj_quantile_Hoeffding_norm} with \eqref{eq:def_quantile_norm} leads straightforwardly to \eqref{eq:maj_quantile_Hoeffding}.  

\paragraph{}
Finally, remains the proof of \eqref{eq:maj_quantile_Hoeffding_norm}. 
There are two cases:
\begin{description}
\item[1\up{st} case:] If $q'_{1-\alpha}(\X_n)\leq 0$, then \eqref{eq:maj_quantile_Hoeffding_norm} is satisfied. 
\item[2\up{nd} case:] If $q'_{1-\alpha}(\X_n) > 0$, then introduce 
$Y=q'_{1-\alpha}(\X_n) - T_\delta'\pa{\X_n^{\rperm}}.$ 

First, since by construction, $\Estar{T_\delta'\pa{\X_n^{\rperm}}}=0$, one directly obtains $\Estar{Y} = q'_{1-\alpha}(\X_n)$. 
Hence, $$0 < q'_{1-\alpha}(\X_n) = \Estar{Y} \leq \Estar{Y\1{Y>0}},$$
and by the Cauchy-Schwarz inequality,  
$$\pa{q'_{1-\alpha}(\X_n)}^2  \quad \leq\quad  \pa{\Estar{Y\1{Y>0}}}^2 \quad 
\leq \quad  \Estar{Y^2} \Estar{\1{Y>0}},$$

Yet, on one hand, 
\begin{eqnarray*}
\Estar{Y^2}
&=& \Estar{\pa{q'_{1-\alpha}(\X_n) - T_\delta'\pa{\X_n^{\rperm}}}^2} \\
&=& \pa{q'_{1-\alpha}(\X_n)}^2 + \Estar{\pa{T_\delta'\pa{\X_n^{\rperm}}}^2} -2 q'_{1-\alpha}(\X_n)\Estar{T_\delta'\pa{\X_n^{\rperm}}} \\ 
&=& \pa{q'_{1-\alpha}(\X_n)}^2 + \Vstar{T_\delta'\pa{\X_n^{\rperm}}} \\
&\leq & \pa{q'_{1-\alpha}(\X_n)}^2 + 1, 
\end{eqnarray*}
since by the normalizing initial step, $\Vstar{T_\delta'\pa{\X_n^{\rperm}}}\leq 1$. 

And, on the other hand, 
\begin{eqnarray*}
\Estar{\1{Y> 0}} &=& \Estar{\1{T_\delta'\pa{\X_n^{\rperm}} < q'_{1-\alpha}(\X_n)}}\\
&=& \frac{\#\ac{\perm\in\Sn{n} \ ;\ T_\delta'\pa{\X_n^{\perm}} < T_\delta^{\prime(n!-\lfloor n!\alpha\rfloor)}\pa{\ds X_n} }}{n!} \\
&\leq & \frac{(n! - \lfloor n!\alpha\rfloor) - 1}{n!} \ =\ 1 - \frac{\lfloor n!\alpha\rfloor+1}{n!} \\
& < & 1-\frac{n!\alpha}{n!} \ =\ 1 - \alpha.
\end{eqnarray*}

So finally, 
$$\pa{q'_{1-\alpha}(\X_n)}^2  \leq \pa{1 - \alpha}\pa{\pa{q'_{1-\alpha}(\X_n)}^2 + 1},$$
which is equivalent to $\pa{q'_{1-\alpha}(\X_n)}^2 \leq  (1-\alpha)/\alpha,$ and thus ends the proof of \eqref{eq:maj_quantile_Hoeffding_norm}.
\end{description}

\subsection{Control of the quantile of the critical value}
\label{sct:control_quantile_du_quantile}

The control of the conditional quantile allows us to upper bound its own quantile $q^\alpha_{1-\beta/2}$ as stated in \eqref{eq:maj_quantile_du_quantile}, that is 
$$q^\alpha_{1-\beta/2} \leq 2\sqrt{\frac{1-\alpha}{\alpha}}\sqrt{\frac{2}{\beta}\frac{\pa{\esps{\indep}{\varphi_\delta^2} + \esps{P}{\varphi_\delta^2}}}{n}}.$$

Indeed, \eqref{eq:maj_quantile_Hoeffding} ensures that
$$q_{1-\alpha}(\X_n) \leq \sqrt{\frac{1-\alpha}{\alpha}} \sqrt{\esp{T_\delta\pa{\X_n^{\rperm}}^2\middle| \X_n}},$$
and in particular, the $(1-\beta/2)$-quantile of $q_{1-\alpha}(\X_n)$ satisfies 
\begin{equation}
\label{eq:maj_quantile_q_quantile_zeta}
q^\alpha_{1-\beta/2} \leq \sqrt{\frac{1-\alpha}{\alpha}}\sqrt{\zeta_{1-\beta/2}}, 
\end{equation}
where $\zeta_{1-\beta/2}$ is the $(1-\beta/2)$-quantile of $\esp{T_\delta\pa{\X_n^{\rperm}}^2\middle| \X_n}$. 
Yet, from Markov's inequality, for all positive $x$, 
$$\proba{\esp{T_\delta\pa{\X_n^{\rperm}}^2\middle| \X_n} \geq x} \leq \frac{\esp{T_\delta\pa{\X_n^{\rperm}}^2}}{x}. $$
In particular, the choice of $x=2\esp{T_\delta\pa{\X_n^{\rperm}}^2}/\beta$ leads to the control of the quantile
\begin{equation}
\label{eq:quantile_zeta_1}
\zeta_{1-\beta/2}\leq \frac{2\esp{T_\delta\pa{\X_n^{\rperm}}^2}}{\beta}. 
\end{equation}

Moreover, noticing that one can write 
$$T_\delta\pa{\X_n^{\rperm}} = \frac{1}{n-1} \sum_{i,j=1}^n \pa{\1{\rperm(i)=j}-\frac{1}{n}}\varphi_\delta(X_i^1,X_j^2),$$
the second-order moment in \eqref{eq:quantile_zeta_1} can be rewritten 
\begin{eqnarray*}
\esp{T_\delta\pa{\X_n^{\rperm}}^2}  
&=& \frac{1}{(n-1)^2} \esp{\pa{\sum_{i,j=1}^n \pa{\1{\rperm(i)=j}-\frac{1}{n}}\varphi_\delta(X_i^1,X_j^2)}^2} \\
&=& \frac{1}{(n-1)^2} \sum_{i,j=1}^n \sum_{k,l=1}^n E_{i,j,k,l} \times\esp{\varphi_\delta(X_i^1,X_j^2) \varphi_\delta(X_k^1,X_l^2)},\\
\end{eqnarray*}
by independence between $\rperm$ and $\X_n$, where $$E_{i,j,k,l} = \esp{\pa{\1{\rperm(i)=j}-\frac{1}{n}} \pa{\1{\rperm(k)=l}-\frac{1}{n}}} = \esp{\1{\rperm(i)=j}\1{\rperm(k)=l}}-\frac{1}{n^2}.$$
On the one hand, for all $1\leq i,j,k,l\leq n$, the Cauchy-Schwarz inequality always ensures
\begin{equation}
\label{eq:maj_varphi_CS}
\esp{\varphi_\delta(X_i^1,X_j^2) \varphi_\delta(X_k^1,X_l^2)} \leq \sqrt{\esp{\varphi_\delta^2(X_i^1,X_j^2)} \esp{\varphi_\delta^2(X_k^1,X_l^2)}} \leq \esps{\indep}{\varphi_\delta^2} + \esps{P}{\varphi_\delta^2},
\end{equation}
since for all $1\leq i,j\leq n$, $\esp{\varphi_\delta^2(X_i^1,X_j^2)}\leq \esps{\indep}{\varphi_\delta^2} + \esps{P}{\varphi_\delta^2}$. 

On the other hand, remains to control the sum $(n-1)^{-2} \sum_{i,j=1}^n \sum_{k,l=1}^n E_{i,j,k,l}$. 
Three cases appear. 
\begin{description}
\item[1\up{st} case: ] If $i\neq k$ and $j\neq l$ (occurring $[n(n-1)]^2$ times), then $$E_{i,j,k,l} = \frac{1}{n(n-1)} - \frac{1}{n^2} = \frac{1}{n^2(n-1)}.$$ 

\item[2\up{nd} case: ] If [$i\neq k$ and $j=l$] or [$i=k$ and $j\neq l$], then $E_{i,j,k,l} = 0-1/n^2 \leq 0.$ 

\item[3\up{rd} case: ] If $i=k$ and $j=l$ (occurring $n(n-1)$ times), then $$E_{i,j,k,l} = \frac{1}{n} - \frac{1}{n^2} = \frac{n-1}{n^2}\leq \frac{1}{n}.$$ 
\end{description}

Therefore, 
\begin{eqnarray}
\frac{1}{(n-1)^2} \sum_{i,j=1}^n \sum_{k,l=1}^n E_{i,j,k,l} 
&\leq & \frac{1}{(n-1)^2}\pa{[n(n-1)]^2 \times\frac{1}{n^2(n-1)} + n(n-1) \times \frac{1}{n}} \nonumber \\
&\leq & \frac{2}{n-1} \nonumber \\
&\leq & \frac{4}{n}. \label{eq:maj_sum_Eijkl}
\end{eqnarray}

Finally, both \eqref{eq:maj_varphi_CS} and \eqref{eq:maj_sum_Eijkl} imply that 
\begin{equation}
\label{eq:esp_carre_2}
\esp{T_\delta\pa{\X_n^{\rperm}}^2} \leq \frac{4}{n}\pa{\esps{\indep}{\varphi_\delta^2} + \esps{P}{\varphi_\delta^2}},
\end{equation}

Therefore, combining \eqref{eq:maj_quantile_q_quantile_zeta}, \eqref{eq:quantile_zeta_1} and \eqref{eq:esp_carre_2} ends the proof of \eqref{eq:maj_quantile_du_quantile}.

\subsection{A first condition ensuing from Hoeffding's approach}
\label{sct:cond_Hoeffding52}

Back to the condition \eqref{eq:cond_simple_tchebychev} derived from Chebychev's inequality, both \eqref{eq:control_variance_gen_gros} and \eqref{eq:maj_quantile_du_quantile} imply that 
$$q^\alpha_{1-\beta/2} + \sqrt{\frac{2}{\beta} \var{T_\delta(\X_n)}} 
\leq \sqrt{\frac{2}{\beta}\frac{\pa{\esps{P}{\varphi_\delta^2} + \esps{\indep}{\varphi_\delta^2}}}{n}} \pa{2\sqrt{\frac{1-\alpha}{\alpha}}+ \sqrt{8}},$$
with $2\sqrt{(1-\alpha)/\alpha}+ \sqrt{8}\leq 4/\sqrt{\alpha}$, since $\sqrt{1-\alpha}+\sqrt{\alpha}\leq \sqrt{2}$.
Finally, the right-hand side of condition \eqref{eq:cond_simple_tchebychev} being upper bounded by 
$$\frac{4}{\sqrt{\alpha}}\sqrt{\frac{2}{\beta}\frac{\pa{\esps{P}{\varphi_\delta^2} + \esps{\indep}{\varphi_\delta^2}}}{n}}, $$ which is exactly the right-hand side of \eqref{eq:cond_Hoeffding}, this ensures the sufficiency of condition~\ref{eq:cond_Hoeffding} to control the second kind error rate by $\beta$.

\paragraph{Acknowledgement}

The author is grateful to Jean-François Coeurjolly for his insightful discussions. This work was supported in part by the French Agence Nationale de la Recherche (ANR 2011 BS01 010 01 projet Calibration). 

\bibliographystyle{alpha}

\begin{thebibliography}{ABFRB16}

\bibitem[ABFRB15]{AlbertBouretFromontReynaud2015}
M.~Albert, Y.~Bouret, M.~Fromont, and P.~Reynaud-Bouret.
\newblock Bootstrap and permutation tests of independence for point processes.
\newblock {\em Annals of Statistics}, 43(6):2537--2564, 2015.

\bibitem[ABFRB16]{AlbertBouretFromontReynaud2016}
M.~Albert, Y.~Bouret, M.~Fromont, and P.~Reynaud-Bouret.
\newblock Surrogate data methods based on a shuffling of the trials for
  synchrony detection: the centering issue.
\newblock {\em Neural Computation}, 28(11):2352--2392, 2016.

\bibitem[ACW14]{AdamczakChafaiWolff2014}
R.~Adamczak, D.~Chafa{\"\i}, and P.~Wolff.
\newblock Circular law for random matrices with exchangeable entries.
\newblock {\em arXiv preprint arXiv:1402.3660}, 2014.

\bibitem[Alb15]{Albert2015PhD}
M\'elisande Albert.
\newblock {\em Tests of independence by bootstrap and permutation: an
  asymptotic and non-asymptotic study. Application to Neurosciences.}
\newblock PhD thesis, Nice, 2015.

\bibitem[Bar02]{Baraud2002}
Y.~Baraud.
\newblock Non-asymptotic minimax rates of testing in signal detection.
\newblock {\em Bernoulli}, 8(5):577--606, 2002.

\bibitem[BDR15]{BercuDelyonRio2015}
B.~Bercu, B.~Delyon, and E.~Rio.
\newblock {\em Concentration inequalities for sums and martingales}.
\newblock Springer, 2015.

\bibitem[BLM13]{BoucheronLugosiMassart2013}
S.~Boucheron, G.~Lugosi, and P.~Massart.
\newblock {\em Concentration inequalities: A nonasymptotic theory of
  independence}.
\newblock Oxford University Press, 2013.

\bibitem[Bol84]{Bolthausen1984}
E.~Bolthausen.
\newblock An estimate of the remainder in a combinatorial central limit
  theorem.
\newblock {\em Probability Theory and Related Fields}, 66(3):379--386, 1984.

\bibitem[Cha07]{Chatterjee2007}
S.~Chatterjee.
\newblock Stein’s method for concentration inequalities.
\newblock {\em Probability theory and related fields}, 138(1):305--321, 2007.

\bibitem[Dwa55]{Dwass1955}
M.~Dwass.
\newblock On the asymptotic normality of some statistics used in non-parametric
  tests.
\newblock {\em The Annals of Mathematical Statistics}, pages 334--339, 1955.

\bibitem[FLRB11]{FromontLaurentReynaud2011}
M.~Fromont, B.~Laurent, and P.~Reynaud-Bouret.
\newblock Adaptive tests of homogeneity for a poisson process.
\newblock {\em Annales de l'Institut Henri Poincaré, Probabilités et
  Statistiques}, 47(1):176--213, 2011.

\bibitem[FLRB13]{FromontLaurentReynaud2013}
M.~Fromont, B.~Laurent, and P.~Reynaud-Bouret.
\newblock The two-sample problem for poisson processes: Adaptive tests with a
  nonasymptotic wild bootstrap approach.
\newblock {\em The Annals of Statistics}, 41(3):1431--1461, 2013.

\bibitem[H{\'a}j61]{Hajek1961}
J.~H{\'a}jek.
\newblock Some extensions of the wald-wolfowitz-noether theorem.
\newblock {\em The Annals of Mathematical Statistics}, 32(2):506--523, 1961.

\bibitem[HC78]{HoChen1978}
S.~T. Ho and L.~H.~Y. Chen.
\newblock An {$L_p$} bound for the remainder in a combinatorial central limit
  theorem.
\newblock {\em The Annals of Probability}, 6(2):231--249, 1978.

\bibitem[Hoe51]{Hoeffding1951}
W.~Hoeffding.
\newblock A combinatorial central limit theorem.
\newblock {\em The Annals of Mathematical Statistics}, 22(4):558--566, 1951.

\bibitem[Hoe52]{Hoeffding1952}
W.~Hoeffding.
\newblock The large-sample power of tests based on permutation of the
  observations.
\newblock {\em The Annals of Mathematical Statistics.}, 23(2):169--192, 1952.

\bibitem[Hoe61]{Hoeffding1961}
W.~Hoeffding.
\newblock The strong law of large numbers for {U}-statistics.
\newblock {\em Institute of Statistics, Mimeograph series}, 302, 1961.

\bibitem[Jog68]{Jogdeo1968}
K.~Jogdeo.
\newblock Asymptotic normality in nonparametric methods.
\newblock {\em The Annals of Mathematical Statistics}, pages 905--922, 1968.

\bibitem[Led05]{Ledoux2005}
M.~Ledoux.
\newblock {\em The concentration of measure phenomenon}, volume~89.
\newblock American Mathematical Soc., 2005.

\bibitem[McD02]{McDiarmid2002}
C.~McDiarmid.
\newblock Concentration for independent permutations.
\newblock {\em Combinatorics, Probability \& Computing}, 11(02):163--178, 2002.

\bibitem[Mot56]{Motoo1956}
M.~Motoo.
\newblock On the {H}oeffding’s combinatrial central limit theorem.
\newblock {\em Annals of the Institute of Statistical Mathematics},
  8(1):145--154, 1956.

\bibitem[MP07]{Massart2007}
P.~Massart and J.~Picard.
\newblock {\em Concentration inequalities and model selection}, volume 1896.
\newblock Springer, 2007.

\bibitem[Noe49]{Noether1949}
G.~E. Noether.
\newblock On a theorem by wald and wolfowitz.
\newblock {\em The Annals of Mathematical Statistics}, pages 455--458, 1949.

\bibitem[Rom89]{Romano1989}
J.~P. Romano.
\newblock Bootstrap and randomization tests of some nonparametric hypotheses.
\newblock {\em The Annals of Statistics}, 17(1):141--159, 1989.

\bibitem[Sch88]{Schneller1988}
W.~Schneller.
\newblock A short proof of motoo's combinatorial central limit theorem using
  stein's method.
\newblock {\em Probability theory and related fields}, 78(2):249--252, 1988.

\bibitem[SH79]{ShapiroHubert1979}
C.~P. Shapiro and L.~Hubert.
\newblock Asymptotic normality of permutation statistics derived from weighted
  sums of bivariate functions.
\newblock {\em The Annals of Statistics}, pages 788--794, 1979.

\bibitem[STM15]{SansonnetTuleau2015}
L.~Sansonnet and C.~Tuleau-Malot.
\newblock A model of poissonian interactions and detection of dependence.
\newblock {\em Statistics and Computing}, 25(2):449--470, 2015.

\bibitem[Tal95]{Talagrand1995}
M.~Talagrand.
\newblock Concentration of measure and isoperimetric inequalities in product
  spaces.
\newblock {\em Publications Math{\'e}matiques de l'Institut des Hautes Etudes
  Scientifiques}, 81(1):73--205, 1995.

\bibitem[WW44]{WaldWolfowitz1944}
A.~Wald and J.~Wolfowitz.
\newblock Statistical tests based on permutations of the observations.
\newblock {\em The Annals of Mathematical Statistics}, 15(4):358--372, 1944.

\end{thebibliography}

\end{document}